\documentclass[12pt]{amsart}

\usepackage{fullpage}
\usepackage{amssymb}
\usepackage{amsmath}
\usepackage[colorlinks, breaklinks, linkcolor=blue]{hyperref}
\usepackage{longtable}

\usepackage{caption}
\newcommand{\C}{\mathbb C}
\newcommand{\F}{\mathbb F}
\newcommand{\Z}{\mathbb Z}
\newcommand{\R}{\mathbb R}

\renewcommand{\a}{\mathfrak a}
\renewcommand{\c}{\mathfrak c}
\newcommand{\g}{\mathfrak g}
\newcommand{\gl}{\mathfrak {gl}}
\newcommand{\h}{\mathfrak h}
\renewcommand{\l}{\mathfrak l}
\newcommand{\m}{\mathfrak m}
\renewcommand{\u}{\mathfrak u}
\renewcommand{\b}{\mathfrak b}

\newcommand{\cO}{\mathcal O}

\newcommand\nn{\mathrm{n}}

\newcommand{\dec}{\mathrm {dec}}
\newcommand{\dist}{\mathrm {dist}}

\DeclareMathOperator{\GL}{GL}
\DeclareMathOperator{\SL}{SL}
\DeclareMathOperator{\Lie}{Lie}
\DeclareMathOperator{\U}{U}

\DeclareMathOperator{\ad}{ad}
\DeclareMathOperator{\Hom}{Hom}
\DeclareMathOperator{\modd}{mod}
\DeclareMathOperator{\RR}{R}
\DeclareMathOperator{\I}{I}
\DeclareMathOperator{\height}{ht}
\DeclareMathOperator{\rank}{rank}
\DeclareMathOperator{\Ker}{ker}
\DeclareMathOperator{\Imag}{im}

\numberwithin{equation}{section}

\newtheorem{thm}[equation]{Theorem}
\newtheorem{prop}[equation]{Proposition}
\newtheorem{lem}[equation]{Lemma}
\newtheorem{assm}[equation]{Assumption}

\theoremstyle{definition}
\newtheorem{defn}[equation]{Definition}

\theoremstyle{remark}
\newtheorem{rem}[equation]{Remark}

\title
{On the coadjoint orbits of maximal unipotent subgroups of reductive groups}

\author[S.~M.~Goodwin]{Simon M. Goodwin}
\address{School of Mathematics, University of Birmingham, Birmingham, B15
2TT, United Kingdom}  \email{s.m.goodwin@bham.ac.uk}

\author[P.~Mosch]{Peter Mosch} \author[G.\ R\"ohrle]{Gerhard R\"ohrle}
\address{Fakult\"at f\"ur Mathematik, Ruhr-Universit\"at Bochum, D-44780 Bochum, Germany}  \email{peter.mosch@rub.de} \email{gerhard.roehrle@rub.de}

 


\begin{document}

\begin{abstract}
Let $G$ be a simple algebraic group
defined over an algebraically closed field
of characteristic $0$ or a good prime for $G$.
Let $U$ be a maximal unipotent
subgroup of $G$ and $\u$ its Lie algebra. We prove the
separability of orbit maps
and the connectedness of centralizers
for the coadjoint action of $U$ on (certain quotients of) the dual $\u^*$ of $\u$.
This leads to a method to give a parametrization of the coadjoint orbits
in terms of so-called {\em minimal representatives}
which form a disjoint union of
quasi-affine varieties. Moreover,
we obtain an algorithm to explicitly calculate this
parametrization which has been used for $G$ of rank at most $8$,
except $E_8$.

When $G$ is defined and split over the field of $q$ elements,
for $q$ the power of a good prime for $G$,
this algorithmic parametrization is used to
calculate the number
$k(U(q), \u^*(q))$
of coadjoint orbits of $U(q)$ on $\u^*(q)$.
Since $k(U(q), \u^*(q))$ coincides with
the number $k(U(q))$
of conjugacy classes in $U(q)$,
these calculations can be viewed as an extension of
the results obtained in \cite{GMR}.
In each case considered here
there is a polynomial $h(t)$ with integer coefficients
such that for every such $q$ we have
$k(U(q)) = h(q)$.  We also explain implications of our results
for a parametrization of
the irreducible complex characters of $U(q)$.
\end{abstract}

\maketitle

\section{Introduction}
\label{sec:intro}

Let $G$ be a simple algebraic group
over an algebraically closed field $k$.
We assume throughout that
$\mathrm{char}(k)$ is either $0$
or a good prime for $G$.
Let $B$ be a Borel subgroup of $G$ and
 $U$ be unipotent radical of $B$.
Further let $\g$, $\b$ and $\u$
be the Lie algebras of $G$, $B$, and $U$.

We study the coadjoint action
of $U$ on the dual $\u^*$ of $\u$.
Our aim is to generalize the theory for the adjoint
action of $U$ on $\u$, developed in
\cite{GOconj}, to the coadjoint action.

Under some mild assumptions that do no harm,
there is a nondegenerate, invariant, symmetric,
bilinear form on $\g$; see \S\ref{subsec:reductp}.
Therefore, we can identify $\g^*$ with $\g$ as a $G$-module.
Under this identification
the annihilator of the $B$-submodule $\u$ of $\g$ is $\b$,
so we can make the identification
of $B$-modules $\u^* \cong \g/\b$.  More generally, we study
quotients of the form $\g/\m$, where $\m$ is a
$B$-submodule of $\g$ containing $\b$; we
restrict to submodules $\m$ that are {\em compatible} with a
certain representation of $G$;
see Definition \ref{D:compatible} for details.
We note that for $G$ of classical type, one can show that
all $B$-submodules $\m$ of $\g$ as above are compatible,
and this is potentially also the case for
$G$ of exceptional type; see Remark \ref{rem:compatible} for more details.

The following are the main results of this paper.

\begin{thm}
\label{thm:sepcoad}
Let $\m \supseteq \b $ be a $B$-submodule of $\g$, which is compatible in the sense of Definition \ref{D:compatible}, and let
$X  + \m \in \g / \m$. Then the orbit map
$U \to U \cdot X + \m$ is separable.
\end{thm}

\begin{thm}
\label{thm:relconn}
Let $\m \supseteq \b $ be a $B$-submodule of $\g$, which is compatible in the sense of Definition \ref{D:compatible}, and let
$X  + \m \in \g / \m$. Then the centralizer $C_U(X + \m)$ of $X+\m$ in $U$ is connected.
\end{thm}

The counterparts to
Theorems \ref{thm:sepcoad} and \ref{thm:relconn}
are known to hold for the adjoint action
of $U$, by previous work of the first author; see
\cite[Prop.\ 3.7, Cor.\ 4.3]{GOconj} for certain quotients of $\u$ and \cite{GOspring} for the
generalization to all quotients; in fact, we remark
that the methods in this paper could be used to give
alternative proofs of those results (for compatible quotients).

Theorem \ref{thm:relconn} allows us to define
\emph{minimal representatives}
of the $U$-orbits in $\u^*$, in analogy with the case for the adjoint action of $U$, as defined in \cite{GOconj}.
The minimal representatives form a disjoint union of quasi-affine
subvarieties of $\u^*$ and parameterize the $U$-orbits in $\u^*$.
As in the adjoint case (see \cite{GMR}), this leads to an algorithm to calculate
a parametrization of the $U$-orbits in $\u^*$.

We have implemented this algorithm in {\sf GAP} \cite{GAP} including some calls to {\sf SINGULAR}
\cite{SIN} and used it to calculate a parametrization of
the $U$-orbits in $\u^*$ for $G$ simple of rank at most $8$,
with the exception of $E_8$.
In this latter case, not unexpectedly,
the computations turn out to
be too complex for the program to
provide such a parametrization.
It would be possible to present the parametrization of the orbits
in detail.  However, this would be rather complicated and
would take up a lot of space. In fact, we can make a good estimate
of the number of families necessary to parameterize the orbits, by the
the sum of the coefficients of the polynomials given in Table \ref{tab:coadjoint}; the meaning
of these polynomials is explained below.
We can see that for the case $G$ of type $E_7$,
tens of thousands of pages would be required.

We mention that in case $k = \C$, Kirillov's method of coadjoint orbits
provides a correspondence between the coadjoint orbits
of $U$ in $\u^*$ and the unitary irreducible representations of $U$,
\cite{KIuni}.
Thus in this case,
our algorithm yields a parametrization of the
equivalence classes of irreducible unitary representations
of $U$.
There is an analogue of Kirillov's orbit method in the case where
$G$ is defined over a finite field with $p$ sufficiently large, which is discussed below.

As mentioned above our algorithm is adapted from the algorithm
we used to parameterize the adjoint
$U$-orbits in $\u$; see \cite{GMR}.
Note that this in turn is based on an algorithm due to
B\"urgstein and Hesselink
which was used to calculate a parametrization of the
$B$-orbits in $\u$ and $\u^*$ for some small rank cases;
see \cite{BUHE}.
We note however, the algorithm devised by
B\"urgstein and Hesselink is not able to
calculate the parametrization of the
$U$-orbits in $\u$ or $\u^*$ in general.
Our procedure to tackle this problem is considerably more complex.

Now assume that $G$ is defined and split over the finite field
$\F_q$ of $q$ elements.  We cite \cite[Sec.\ 3]{DM}, as a reference
for the theory of reductive groups over finite fields, and for the
terminology used here.  Consider the
finite Chevalley group $G(q)$ of $\F_q$-rational points of $G$.
In this case, we also investigate the action of the Sylow $p$-subgroup $U(q)$
of $G(q)$ on $\u^*(q)$.

In the special case when $G = \GL_n(k)$,
let $U = \U_n(q)$ be the subgroup of $\GL_n(q)$ consisting of upper
unitriangular matrices, where $q$ is a power of a prime. A
longstanding conjecture attributed to G.~Higman (cf.\ \cite{higman})
states that the number of conjugacy classes of $\U_n(q)$ for fixed
$n$ is a polynomial in $q$ with integer coefficients. This
conjecture has been verified for $n \le 13$ by computer calculation
in work of A.~Vera-Lopez and J.~M.~Arregi, \cite{VLAR}.
There has been considerable interest in this
conjecture, for example from G.~R.~Robinson \cite{robinson} and
J.~Thompson \cite{thompson}.  We remark here that a recent
paper of Z.~Halasi and P.~P.~P\`alfy casts some doubt on the
conjecture, \cite{HP}.

Of course, Higman's conjecture
can be stated in terms of the irreducible complex characters of $\U_n(q)$.
There has been much interest in approaching this question from this direction,
which we discuss later.

It is natural to
consider the analogue of Higman's conjecture for other finite
Chevalley groups. This investigation started in
\cite{GOROconj} and continued in \cite{GMR}.
Our new results in this direction (see Section \ref{sec:applic})
can be viewed as evidence for
this analogous conjecture.

By a trivial generalization of \cite[Thm.\ 2.1.2]{KIvar}, it is easy to see that
the number $k(U(q), \u(q))$ of
$U(q)$-orbits in $\u(q)$
and the number
$k(U(q), \u^*(q))$ of
$U(q)$-orbits in $\u^*(q)$ coincide.  Further, $k(U(q),\u(q)) = k(U(q))$ (see for example
\cite[Prop.\ 6.2]{GOconj}), so that $k(U(q), \u^*(q)) = k(U(q))$.
The parametrization of the $U$-orbits in $\u^*$
can be used to adapt the algorithm employed in \cite{GMR}
to calculate the number $k(U(q), \u^*(q))$
of $U(q)$-orbits in $\u^*(q)$ for $G$ simple
of rank at most $8$,
with the exception of $E_8$; see Section \ref{sec:applic}.
Combining our results with
\cite[Thm.\ 1.1]{GOROconj} and
\cite[Thm.\ 1.1]{GMR}, we obtain the following theorem.

\begin{thm}
\label{thm:poly}
Let $G$ be a split simple algebraic group defined over $\F_q$ of rank at most
$8$, not of type $E_8$, where $q$ is a power of a good prime.
Let $U$ be a maximal unipotent subgroup of $G$ which is also
defined over $\F_q$. Then there is a polynomial $h(t) \in \Z[t]$ which only depends on the
Dynkin type of $G$ such that the number $k(U(q))$
of conjugacy classes in $U(q)$ is $h(q)$.
Furthermore, if one considers $k(U(q))$ as a
polynomial in $q-1$, then the coefficients are non-negative.
\end{thm}

The polynomials giving $k(U(q))$ are presented in the tables in \cite{VLAR}, \cite{GOROconj}, \cite{GMR} and in
Table \ref{tab:coadjoint} of this paper.
Also we remark that the restriction to good primes
is necessary to get the same polynomial $h(t)$ as the prime varies,
as observed in \cite{BRGO}.

As mentioned above, Higman's conjecture can also be approached by parameterizing the irreducible
complex characters of $\U_n(q)$.  This has generated a great deal
of interest. We give a brief (incomprehensive) overview here.  In \cite{Le},
G.~Lehrer considered the characters of $\U_n(q)$ and in
particular conjectured that the degrees of the irreducible characters are always
a power of $q$, and that
the number of characters of degree $q^d$, for $d \in \Z_{\ge 0}$ is given by a polynomial in $q$ with integer coefficients.
In \cite{Is}, I.~M.~Isaacs showed that the degree of an irreducible character of $\U_n(q)$
is always a power of $q$; in fact this was proved more generally for the class of algebra groups.
We denote the number of irreducible complex characters of $\U_n(q)$ of degree $q^d$ by
$k(\U_n(q),q^d)$.
More recently, Isaacs has refined Lehrer's conjecture to say that $k(\U_n(q),q^d)$ viewed as a polynomial in
$q-1$ has positive integer coefficients.
Further recent progress has also been, for example, in \cite{EV}, A.~Evseev verified that Isaacs' conjecture
holds for $n \le 13$, and, in \cite{Ma}, E.~Marberg showed that Isaacs' conjecture
holds for $d \le 8$.

It is interesting to consider the analogues of Lehrer's and Isaacs' conjectures
for other finite Chevalley groups.  Our results on the coadjoint orbits
make progress on this.  To explain this we require the version of the Kirillov orbit method
for $U(q)$, which we briefly describe.

The Kirillov orbit method leads to a bijection
between the coadjoint orbits of $U(q)$ in $\u^*(q)$ and the complex
irreducible characters of $U(q)$, provided $p$ is sufficiently large.
Let $h$ denote the Coxeter number of $G$ and suppose that $p \ge h$.  Then
there is an exponential map from $\u$ to $U$ and a logarithm map from $U$ to $\u$
such that the Baker--Campbell--Hausdorff formula holds, see
\cite[Thm.\ 3]{Se}.  Therefore, \cite[Prop.\ 2]{KAZ}
provides a means to associate a complex character $\chi_\cO$ of $U(q)$ to a coadjoint orbit $\cO$
of $U(q)$ in $\u^*(q)$.  To do this we first have to fix a nontrivial character $\theta : \F_q \to \C^\times$.
Then the character $\chi_\cO$ is defined by the formula
$$
\chi_\cO(u) = \frac{1}{\sqrt{|\cO|}}
\sum_{F \in \cO} \theta(F(\log(u))).
$$
Then $\cO \mapsto \chi_\cO$ is the desired bijection.
Hence, in case $p \ge h$ our algorithm gives a parametrization of the complex characters
of $U(q)$.
We refer the reader also to \cite[Sec.~1]{BOSA} for a more recent and alternative proof of these results,
as well as versions of the Kirillov orbit method applicable in other situation; and also to \cite{Sa} for a version
applicable to algebra groups.

As a consequence of the formula defining $\chi_\cO$ above, $\chi_\cO(1) = \sqrt{|\cO|}$.
By Proposition \ref{P:orbitsize}, the size of any coadjoint
orbit of $U(q)$ in $\u^*(q)$ is equal to $q^{2d}$ for some $d \in \Z_{\ge 0}$.  Thus it
follows that (under the assumption that $p \ge h$) the degrees of all irreducible characters of $U(q)$
are equal to $q^d$ for some $d \in \Z_{\ge 0}$.
We denote the number of irreducible complex characters of $U(q)$ of degree $q^d$ by
$k(U(q),q^d)$.  As explained in Section \ref{sec:applic}, our calculations verify that
the number of coadjoint orbits of $U(q)$ in $\u^*(q)$ of size $q^{2d}$ is given by a polynomial in $q$,
for $G$ of rank less than or equal to 8, but not of type $E_8$.

Hence, the analogue of Lehrer's conjecture holds in these cases, and in fact we can
observe that the analogue of Isaacs' conjecture holds in these cases as well from the table
in the appendix.

It is interesting to know if the degrees of all irreducible characters of $U(q)$ are
powers of $q$. For $G$ of classical type it was shown by Sangroniz in \cite{Sa2} that the degrees of all irreducible characters of $U(q)$
are powers of $q$ if and only if $p \ne 2$.
For $G$ of type $E$ it is shown by Le and Magaard in \cite{LM}
that the character degrees are not all powers of $q$ for bad primes.  This
is also likely to be the case for $G$ of type $G_2$, where it is
fairly straightforward to calculate all character degrees, and also for $G$ of
type $F_4$.
This leaves exceptional types for $p$ good and $p < h$, where
we expect the degrees of all irreducible characters of $U(q)$ to be powers of $q$.

 As explained above, the number
$k(U(q), \u^*(q))$ of
$U(q)$-orbits in $\u^*(q)$ is equal to $k(U(q))$.
However, it appears to be an interesting question as to whether there is an explicitly defined bijection from the coadjoint orbits to
the irreducible characters.
It is shown by Marberg in \cite{Ma2} that the formula for $\chi_\cO$ above does not always give
an irreducible character of $\U_n(q)$.  We refer the reader to the introduction of \cite{Ma} for
a discussion about this question for $\U_n(q)$.
We note that there has been related recent work regarding
the parametrization of the complex characters for Sylow $p$-subgroups of
finite groups of Lie type; see for example \cite{HH}, \cite{HLM} and \cite{Lee}.

We mention that one can get various
parabolic analogues of our results, and in particular of Theorem \ref{thm:poly};
see Remark \ref{rem:kup} and
\cite[\S 5.1]{mosch:phd} for details.

To end the introduction, we mention that
in Section \ref{subsect:modality} we
present some new results on the modality of the action of
$B$ on $\u$ which can be derived from our computations.
More specifically, we are able
to determine the modality of the action of $B$ on
$\u$ for instances where previously only lower bounds were known.
This gives a significant extension of the results presented in
\cite[Tables II and III]{jurgensroehrle}, and further demonstrates the strength of our algorithm.

\section{Preliminaries}
\label{sec:prelim}

\subsection{Basic Notation}
\label{subsec:notation}

Let $k$ be an algebraically closed field.
Given an algebraic group $H$ over $k$, we write $H^\circ$
for the identity component of $H$, and we denote the Lie
algebra of $H$ by $\h$.  When $H$ acts on an algebraic
variety $V$ and $x \in V$, we write $C_H(x)$ for the centralizer
of $H$ in $V$, and $H \cdot x$ for the $H$-orbit of $x$.  If $W$ is a subset of $V$, we write $C_H(W)$ for the
(pointwise) centralizer of $W$ in $H$.  In the case, where $V$ is a rational
$H$-module, then $V$ is also an $\h$-module, and for $x \in V$ we write $\c_\h(x)$ for the centralizer
of $x$ in $\h$.

Throughout, $G$ is a
simple algebraic group
over $k$, where the
characteristic $p$ of $k$ is either $0$ or a good prime for $G$; for convenience we also allow
$G = \GL_n(k)$.
Let $B$ be a Borel subgroup of $G$ and write $U$ for the unipotent
radical of $B$.
Let $T$ be a maximal torus of $G$ contained in $B$ so that
$B = TU$.

\subsection{Compatible submodules}
\label{subsec:reductp}

Let $V$ be a faithful rational $G$-module.  This allows us to identify $G$ as a subgroup
of $\GL(V)$ and $\g$ as a subalgebra of $\gl(V)$.  We often want to have the following assumption.

\begin{assm} \label{assm:1}
The restriction
of the trace form on $\gl(V)$ to $\g$ is nondegenerate.
\end{assm}

For $G$ a classical group,
we note that Assumption \ref{assm:1} holds if we take $V$ to be the natural
module; except in the case $G = \SL_n(k)$ and $p \mid n$, in which case we can ``replace''
$\SL_n(k)$ by $\GL_n(k)$. For $G$ of
exceptional type, we can take $V = \g$ to be the adjoint module.
See \cite{rich2} or \cite[Ch.\ I, Lem.\ 5.3]{SPST} for details.
As explained at the end of the proof of \cite[Thm.\ 3.9]{GOconj}, we have that $U$ is independent
up to isomorphism of the isogeny class of $G$, so our results on the coadjoint
action of $U$ do not depend on this assumption.

For the remainder of this subsection we make Assumption \ref{assm:1}.  We let $\tilde \g$ be the
orthogonal complement in $\gl(V)$ of $\g$ with respect to the trace form.
Then $\tilde \g$ is a $G$-submodule of $\gl(V)$ under the adjoint action, and moreover
$\gl(V)$ admits a $G$-module decomposition
\begin{equation} \label{eq:redpair}
\gl(V) = \g \oplus \tilde \g.
\end{equation}
In other words, this means that $(\GL(V),G)$ is a \emph{reductive pair},
as defined in \cite{rich2}.
We use the notation $\bar G = \GL(V)$ and $\bar \g = \gl(V)$.

We are interested in $B$-submodules $\m$ of $\g$ containing $\b$ that are compatible with the
direct sum decomposition of \eqref{eq:redpair} as set out in the next definition.

\begin{defn}
\label{D:compatible}
Make Assumption \ref{assm:1}, and let $\bar \g = \g \oplus \tilde \g$
be the corresponding $G$-module decomposition from \eqref{eq:redpair}.
Let $\m$ be a $B$-submodule of $\g$ containing $\b$.
We say that $\m$ is {\em compatible (with $V$)}
provided there exist a Borel subgroup
$\bar B$ of $\bar G$, a
$\bar B$-submodule $\bar \m$ of $ \bar \g$,
and a $B$-submodule
$\tilde \m \subseteq \tilde \g$  such that
\begin{itemize}
\item[(i)] $B = \bar B \cap G$;
\item[(ii)] $\bar \m \supseteq \bar \b$; and
\item[(iii)] $\bar \m = \m \oplus \tilde \m$ as $B$-modules.
\end{itemize}
\end{defn}

We note that a consequence of (iii) is that there is an isomorphism of $B$-modules
\begin{equation} \label{e:id}
\bar \g/ \bar \m \cong \g/\m \oplus \tilde \g/\tilde \m.
\end{equation}

\begin{rem}
\label{rem:compatible}
It is quite an easy exercise to see that for $G$ of classical type,
all $B$-submodules $\m$ of $\g$ are compatible with the natural representation;
a little care needs to be taken with regards to the graph automorphism
in type $D$.  Below we give a general method for constructing compatible
$B$-submodules $\m$ of $\g$; it seems plausible that all $\m$ can be obtained in
this way, though we do not pursue this here.
\end{rem}

First we introduce some more notation.
Denote the character group of $T$ by  $X(T) = \Hom(T,k^\times)$.
Let $\Phi \subseteq X(T)$ be the root system of $G$ with respect to $T$
and let $\Phi^+$ be the set of positive roots determined by $B$, and $\Phi^- = - \Phi^+$.
Let $\Pi$ be the set of simple roots in $\Phi^+$.

We use certain preorders on $X(T)$ to construct compatible $B$-submodules, as defined next.

\begin{defn} \label{D:prec}
We call a total preorder $\preceq$ on $X(T)$ a {\em $B$-preorder} if
it satisfies the following three properties
for all $\beta,\gamma,\beta',\gamma' \in X(T)$.
\begin{itemize}
\item[(i)] If $\beta \in \Phi^+$, then $0 \preceq \beta$.
\item[(ii)] If $\beta \preceq \gamma$ and $\beta' \preceq \gamma'$,
then $\beta + \beta' \preceq \gamma + \gamma'$.
\item[(iii)] If $\beta \preceq \gamma$, then $-\gamma \preceq -\beta$.
\end{itemize}
\end{defn}

An example of a $B$-preorder, which is in fact a total order, can be given as follows.
Fix an enumeration $\Pi = \{ \alpha_1, \ldots, \alpha_r \}$ of $\Pi$, and define a total order
$\preceq$ on $X(T)$ as follows for $\beta, \gamma \in X(T)$:
\begin{itemize}
\item If $\height \beta < \height \gamma$, then $\beta \prec \gamma$.
\item If $\height \beta = \height \gamma$ and
$\beta = \sum_{i=1}^{r} b_i \alpha_i, ~ \gamma = \sum_{i=1}^{r} c_i \alpha_i$ where
$b_j < c_j$ for the highest $j$ such that $b_j \ne c_j$, then $\beta \prec \gamma$.
\end{itemize}
Here we recall that $\height(\sum_{i=1}^{r} b_i \alpha_i) := \sum_{i=1}^{r} b_i$.

Also, we can obtain $B$-preorders from a function $d : X(T) \to \R$ such that
$d(\alpha + \beta) = d(\alpha) + d(\beta)$ for all $\alpha, \beta \in X(T)$, and $d(\alpha) \ge 0$ for
all $\alpha \in \Phi^+$.  We note that such a function $d$ is uniquely determined from its restriction to $\Pi$, and
that any function $d : \Pi \to \R_{\ge 0}$ can be obtained as such a restriction.
Given such $d : X(T) \to \R$, we can define the $B$-preorder $\preceq$ by
$\alpha \preceq \beta$ if and only if $d(\alpha) \le d(\beta)$.

For the rest of this subsection we fix a $B$-preorder $\preceq$.

For $\beta \in \Phi$, let $\g_{\beta}$ be the corresponding root space of $\g$.
Given $\gamma \in \Phi^-$, we define the subspaces $\m_{\succ \gamma}$ of $\g$ by
\begin{equation} \label{e:m_beta}
\m_{\succ \gamma} := \bigoplus_{\beta \succ \gamma} \g_{\beta}.
\end{equation}
Our choice of order $\prec$ ensures that each $\m_{\succ \gamma}$ is a $B$-submodule
of $\g$ containing $\b$.

Recall the faithful $G$-module $V$ from above.
We continue to make the Assumption \ref{assm:1}.
We let $\Psi(V) \subseteq X(T)$ be the set of
$T$-weights in $V$ and decompose
$$
V = \bigoplus_{\nu \in \Psi(V)} V_\nu
$$
as a direct sum of its $T$-weight spaces.
We choose an (ordered) basis $\{v_1,\dots,v_n\}$ of $V$ consisting of $T$-weight vectors,
where the $T$-weight of $v_i$ is $\nu_i$, and we have $\nu_i \preceq \nu_j$, whenever $i \ge j$.
Using this basis we can identify $\bar G$ with $\GL_n(k)$.  We then let $\bar B$ be the upper triangular
matrices under this identification, and note that our choice of basis ensures that $B = \bar B \cap G$.

Next we define $\Psi(\bar \g) \subseteq X(T)$ to be the set of weights of $T$ on $\bar \g$.
We note that we have
$$
\Psi(\bar \g) = \{\nu - \mu \mid \nu, \mu \in \Psi(V)\}.
$$
Then we decompose $\bar \g = \gl_n$ into $T$-weight spaces
$$
\bar \g = \bigoplus_{\lambda \in \Psi(\bar \g)} \bar \g_\lambda.
$$
Since $\bar \g = \g \oplus \tilde \g$ is a $G$-module direct sum
decomposition, we obtain $\bar \g_\lambda = \g_\lambda \oplus \tilde \g_\lambda$
for all $\lambda \in \Psi(\bar \g)$.

The following lemma ensures our supply of compatible $B$-submodules of $\g$.

\begin{lem}
\label{lem:barem}
Let $\gamma \in \Phi^-$, and define $\m = \m_{\succ \gamma}$ as in \eqref{e:m_beta}.  Then $\m$ is compatible.
More precisely, if one chooses $\bar B$ as above, and
$$
\bar \m = \bar \m_{\succ \gamma} = \bigoplus_{\lambda \succ \gamma} \bar \g_\lambda \quad \text{and} \quad \tilde \m = \tilde \m_{\succ \gamma} = \bigoplus_{\lambda \succ \gamma} \tilde \g_\lambda
$$
then properties (i) -- (iii)
from Definition \ref{D:compatible} hold for $\bar B$, $\bar \m$ and $\tilde \m$.
\end{lem}

\begin{proof}
It is clear that $\bar \m \supseteq \bar \b$, and that $\bar \m = \m \oplus \tilde \m$.  The main property to prove
is that $\bar \m$ is a $\bar B$-submodule of $\bar \g$; from this it follows easily that $\tilde \m$ is a $B$-submodule of $\tilde \g$.

Through the basis $\{v_1,\dots,v_n\}$ of $V$ we have our
identification of $\bar \g$ with $\gl_n(k)$ and we write $\{e_{ij} \mid 1 \le i,j \le n\}$
for its standard basis.  For $\lambda \in \Psi(\bar \g)$, we have
$\bar \g_\lambda = \langle e_{ij} \mid \lambda = \nu_i - \nu_j \rangle$.
By considering how matrices are multiplied together, we see that it suffices to prove the following:
Let $1 \le i,j,k,l \le n$ and suppose that $k \le i$ and $j \le l$, then $\nu_i - \nu_j \preceq \nu_k - \nu_l$.
Well, since $k \le i$, we have $\nu_i \preceq \nu_k$, and since $j \le l$, we have $\nu_l \preceq \nu_j$.  Therefore, $-\nu_j \preceq -\nu_l$, and so
$\nu_i-\nu_j \preceq \nu_k-\nu_l$, which is what we require, so the proof is complete.
\end{proof}

\subsection{Springer isomorphisms}
\label{subsec:springer}

Denote by $\mathcal{U}$ the unipotent variety of $G$
and by $\mathcal{N}$ the nilpotent variety of $\g$.
In \cite{Springer}, Springer showed
that, if $G$ is simply connected,
then there is a $G$-equivariant
isomorphism of varieties
\[
\phi: \mathcal{U} \rightarrow \mathcal{N};
\]
see also \cite[Thm.\ 3.12]{SPST}.
Such an isomorphism is called a \emph{Springer isomorphism}.
We note that for $G$ not necessarily simply connected, we still obtain
a $G$-equivariant bijective morphism of varieties $\phi: \mathcal{U} \rightarrow \mathcal{N}$,
which is an isomorphism if the covering map from the simply connected cover of $G$ is separable.
See, for example \cite[\S 2]{GOconj}, for an overview of the theory of Springer isomorphisms.

We recall the construction of a Springer isomorphism due to
Bardsley and Richardson \cite[\S 9]{BARI}, in the case where Assumption \ref{assm:1} holds.
For that recall the $G$-module decomposition of $\bar \g = \gl(V)$
from \eqref{eq:redpair}.
Let $\iota: G \to \GL(V)$ be the embedding of $G$ into $\GL(V) \subseteq \gl(V)$
and $\pi: \gl(V) \to \g$ be the canonical projection.
It follows from the proof of \cite[Lem.\ 9.3.1]{BARI}
that under the assumptions made,
we obtain a Springer isomorphism
by restricting $\pi \circ \iota$ to the unipotent variety $\mathcal{U}$ of $G$,
i.e.\ the image of $\pi \circ \iota|_{\mathcal U}$ is precisely $\mathcal N$, and this
restriction is an isomorphism of varieties onto $\mathcal N$.

\section{The coadjoint action of $U$ on $\u^*$}
\label{sec:coad}

We continue to make Assumption \ref{assm:1}, so that there
is a nondegenerate, symmetric, invariant, bilinear form on $\g$.  As
explained in the introduction this allows us to identify $\u^*$
with $\g/\b$ as $U$-modules.
Let $\m \supseteq \b$ be a compatible $B$-submodule of $\g$, and let $\bar \m$, $\tilde \m$
and $\bar B$ be as in Definition \ref{D:compatible}, and write $\bar U$ for the unipotent radical of $\bar B$.
Then we have an isomorphism of $B$-modules from \eqref{e:id}, which we use throughout this section.

In this section, we prove Theorems \ref{thm:sepcoad} and \ref{thm:relconn} regarding the action of
$U$ on $\g/\m$.  Our
strategy is to adapt and extend some arguments of Springer and Steinberg from \cite[Ch.\ I \S5]{SPST}.

Let $X + \m \in \g / \m$.
The $U$-orbit of $X+\m$ in $\g / \m$ is denoted by $U \cdot X+\m$ and
we write $\bar U \cdot X + \bar \m$
for the $\bar U$-orbit of $X+ \bar \m$ in $\bar \g / \bar \m$.
We define
\[
V(X,\bar \m) := \{u \cdot X - X + \bar \m \in \bar \g/\bar \m \mid u \in \bar U\}.
\]

\begin{lem}
\label{lem:dfsurj}
Define the map $f : \bar U \to V(X,\bar \m)$ by
\[
f(u) = (u \cdot X) - X + \bar \m.
\]
Then $f$ is separable.
\end{lem}

\begin{proof}
We show that the differential
\[
(df)_1:
\bar \u \to T_0(V(X,\bar \m))
\]
of $f$ is surjective.
A direct calculation shows that $(df)_1$
maps $Y \in \bar \u$ to $[Y,X] + \bar \m$. We have
\[
\dim T_0(V(X,\bar \m))
= \dim (\bar U \cdot X + \bar \m)
= \dim \bar U - \dim C_{\bar U}(X + \bar \m),
\]
while
\[
\dim \Imag(df)_1 = \dim \bar \u - \dim \Ker(df)_1
= \dim \bar U - \dim \c_{\bar \u}(X + \bar \m).
\]
Thus, it suffices to show that
$\dim C_{\bar U}(X + \bar \m) = \dim \c_{\bar \u}(X + \bar \m)$.
We claim that in fact
\begin{equation} \label{eq:cu}
C_{\bar U}(X + \bar \m) = 1 + \c_{\bar \u}(X + \bar \m).
\end{equation}
If $u \in C_{\bar U}(X + \bar \m)$,
then $u X u^{-1} + \bar \m = X + \bar \m$.
Since $\bar \m$ is a $\bar B$-submodule of $\bar \g$, we have that $\bar \m$ is
stable under right multiplication by elements of $\bar B$.
Thus right multiplication by $u$ and subtraction of $Xu$ on both sides yields
\begin{align*}
\bar \m & = u X - X u + \bar \m \\
 & = (u - 1) X - X (u -1) + \bar \m \\
 & = [u - 1, X] + \bar \m.
\end{align*}
So $C_{\bar U}(X + \bar \m) \subseteq 1 + \c_{\bar \u}(X + \bar \m)$.
The reverse inclusion can be proved similarly,
and so \eqref{eq:cu} follows.
\end{proof}

\begin{lem}
\label{lem:separcent}
$\Lie C_U(X + \m) = \c_\u(X + \m)$.
\end{lem}

\begin{proof}
Since $\Lie C_U(X + \m) \subseteq \c_\u(X + \m)$,
it suffices to show that the dimensions coincide.
This is equivalent to showing that $\dim (U \cdot X + \m) = \dim ( [\u,X] + \m)$, or equivalently
\begin{equation} \label{eq:dimad}
\dim (U \cdot X + \m) = \dim (\ad(X)(\u) + \m).
\end{equation}
Consider the restriction of $f$ from Lemma \ref{lem:dfsurj} to $U$.
Through \eqref{e:id}, we have that the image of $f|_U$ is contained
in $\g/\m$, and thus also $T_0(f(U)) \subseteq \g / \m$.
We prove \eqref{eq:dimad}
by verifying the following sequence of inclusions:
\begin{equation}
\begin{aligned} \label{eq:tangent}
T_0(f(U))
 & \subseteq T_0(f(\bar U)) \cap T_0(\g / \m) \\
 & = (df)_1(\bar \u) \cap \g / \m \\
 & = (\ad(X)(\bar \u) + \bar \m) \cap \g / \m \\
 & = \ad(X)(\u) + \m \\
 &  \subseteq T_0(f(U)).
\end{aligned}
\end{equation}
The initial inclusion is immediate,
whilst the first equality is just Lemma \ref{lem:dfsurj} and the second
is the definition of $(df)_1$.

To see why the third equality holds,
consider $\bar{Y} = Y + \tilde{Y} \in \bar \u$
with $Y \in \u$ and $\tilde{Y} \in \u \cap \tilde \g$.
Then
\[
[X,\bar{Y}] + \bar \m = [X, Y + \tilde{Y}] + \bar \m
= [X, Y] + [X, \tilde{Y}] +  \bar \m,
\]
where, thanks to the $G$-module decomposition of $\g$
from \eqref{eq:redpair}, we have that
$[X, Y]$ belongs to $\g$ and
$[X, \tilde{Y}]$ lies in $\tilde \g$.
Thus through \eqref{e:id}, we have that $[X, Y] + \m \in \g / \m$
and $[X, \tilde{Y}] + \tilde \m \in \tilde \g / \tilde \m$.
In particular,
if $[\bar{Y}, X] + \bar \m \in \g / \m$,
then $[\bar{Y},X] +\bar \m = [Y,X] + \bar \m$.  This shows that
$(\ad(X)(\bar \u) + \bar \m) \cap \g / \m \subseteq \ad(X)(\u) + \m$, and the
reverse inclusion is clear.

The last inclusion holds
because $\ad(X)(\u)  + \m = (df)_1(\u)$.

Finally, \eqref{eq:tangent} yields
that $T_0(f(U)) = \ad(X)(\u) + \m$,
and, clearly, we also have that $\dim f(U) = \dim (U \cdot X + \m)$,
which together imply \eqref{eq:dimad}.
\end{proof}

We are now in a position to prove our main results.

\begin{proof}[Proof of Theorem \ref{thm:sepcoad}]
By a standard argument this is equivalent to Lemma \ref{lem:separcent}.
\end{proof}

\begin{proof}[Proof of Theorem \ref{thm:relconn}]
Let $\phi: \mathcal{U} \rightarrow \mathcal{N}$ be the
Springer isomorphism for $G$ obtained from the Bardsley--Richardson
construction, as in Section \ref{subsec:springer}.
We prove the statement by showing that
$C_U(X + \m) = \phi^{-1}(\c_\u(X + \m))$.
Since $\c_\u(X + \m)$ is a vector space and thus is irreducible as
a variety, so is  $\phi^{-1}(\c_\u(X + \m))$.
Then it follows that $C_U(X + \m)$ is connected.

We identify $U$ as a subgroup of $\bar U$.
Thus, if $u \in C_U(X + \m)$,
then $u \in C_{\bar U}(X + \bar \m)$, so $u X u^{-1} + \bar \m = X + \bar \m$.
Since $\bar \m$ is a $\bar B$-submodule of $\bar \g$, it is invariant under
right multiplication
by elements in $\bar U$. It follows that $u X + \bar \m = X u + \bar \m$.
Therefore, $[u, X] + \bar \m = \bar \m$.
The direct sum decomposition \eqref{eq:redpair}
of $\bar \g$ and the definition of $\phi$ gives
\[
u = \phi(u) + \tilde{Y}
\]
with $\tilde{Y} \in \tilde \g$.  So we have
$$
[u, X] + \bar \m = [\phi(u), X] + [\tilde{Y}, X] + \bar \m
$$
where $[\phi(u), X] \in \g$ and $[\tilde{Y}, X] \in \tilde \g$.
Thus, using \eqref{e:id}, we obtain
$[\phi(u), X] \in \m$ and $[\tilde{Y}, X] \in \tilde \m$.
Together with the fact that $\phi$ restricts
to an isomorphism $U \cong \u$
(see for example \cite[Cor.\ 2.2]{GOconj}),
this implies that $\phi(u) \in \c_\u(X + \m)$.

Consequently, $\phi(C_U(X + \m)) \subseteq \c_\u(X + \m)$.
Since $\phi(C_U(X + \m))$ is a closed subvariety of $\c_\u(X + \m)$,
both have the same dimension,
by Lemma \ref{lem:separcent}, and $\c_\u(X + \m)$ is irreducible,
we must have that $\phi(C_U(X + \m)) = \c_\u(X + \m)$.
Finally, as $\phi$ is
an isomorphism,
$C_U(X + \m) = \phi^{-1}(\c_\u(X + \m))$, as desired.
\end{proof}

We now explain how Theorems \ref{thm:sepcoad} and \ref{thm:relconn} allow
us to give a parametrization of the coadjoint $U$-orbits in $\u^*$; this is
an analogue of the parametrization of the adjoint $U$-orbits in $\u$
described in \cite{GOROconj}.  We first fix a $B$-preorder $\preceq$ on $X(T)$,
as defined in Definition \ref{D:prec},
and assume that $\preceq$ restricts to a total order on $\Phi$.
Now we enumerate $\Phi^+ = \{\beta_1,\dots,\beta_N\}$ so that
$\beta_i \prec \beta_j$, whenever $i < j$.
To ease notation, we also define $\gamma_i = -\beta_{N+1-i}$, so that
$\Phi^- = \{\gamma_1,\dots,\gamma_N\}$.

We define subspaces
$\m_i$ of $\g$, for each $0 \leq i \leq N$, by setting
\begin{equation} \label{e:m_i}
\m_i := \m_{\succ \gamma_i} = \left(\bigoplus_{j=i+1}^{N} \g_{\gamma_j}\right) \oplus \b \subseteq \g.
\end{equation}
Then each $\m_i$ contains $\b$ and is a compatible $B$-submodule of $\g$, by Lemma \ref{lem:barem}.
Thus the quotient $\g / \m_i$ is a $B$-module as well. Furthermore, we have
\[
\b = \m_N  \subset \ldots  \subset \m_i \subset \ldots  \subset \m_0 = \g
\]
and $\dim \m_i / \m_{i-1} = 1$ for $i = 1, \ldots, N$.

Using Theorem \ref{thm:relconn} one can prove
the following lemma.  It is an analogue of \cite[Lem.\ 5.1]{GOconj}, and the proof
follows the same line of arguments; we include it for the reader's convenience.  In the statement
$e_\beta$ denotes a generator of the root space $\g_\beta$, for $\beta \in \Phi$.

\begin{lem}
\label{lem:raminco}
For $0 \leq i \leq N$, let $\m_i$ be defined as in \eqref{e:m_i}.
Let $X + \m_{i-1}$ be an element in $\g / \m_{i-1}$.
Denote $X + k e_{\gamma_i} + \m_i = \{ X + \lambda e_{\gamma_i} + \m_i \mid \lambda \in k \} \subseteq \g / \m_i$. Then either
\begin{enumerate}
\item[(I)] all elements of $X+k e_{\gamma_i}+\m_i$ are $U$-conjugate, or
\item[(R)] no two elements of $X+k e_{\gamma_i}+\m_i$ are $U$-conjugate.
\end{enumerate}
\end{lem}

\begin{proof}
Let $\lambda \in k$, and consider $\mathcal Y = C_U(X+\m_{i-1}) \cdot (X+\lambda e_{\gamma_i}+ \m_i)$.  By
definition of $C_U(X+\m_{i-1})$ and $X + k e_{\gamma_i} + \m_i$, we have that $\mathcal Y \subseteq X + k e_{\gamma_i} + \m_i$.  Since $C_U(X+\m_{i-1})$ is unipotent, we have that $\mathcal Y$
is closed in $X + k e_{\gamma_i} + \m_i$; see for example \cite[Prop.\ 4.10]{Bo}.  Further,
$\mathcal Y$ is connected, because
$C_U(X+\m_{i-1})$ is connected.  Therefore, $\mathcal Y$
is a closed and connected
subset of $X + k e_{\gamma_i} + \m_i \cong k$.  As such either $\mathcal Y$
is equal to $X+k e_{\gamma_i}+\m_i$, so that (I) holds, or
$\mathcal Y$ is equal to $X+\lambda e_{\gamma_i}+ \m_i$ for all $\lambda \in k$, so that (R) holds.
\end{proof}

In case (I) of the above lemma, we say that $i$ is an \emph{inert point} of $X + \m_i$, whereas
in case (R) we call $i$ a \emph{ramification point} of $X + \m_i$.
We say that $X + \m_i = \sum_{j=1}^i a_j e_{\gamma_j} + \m_i \in \g/\m_i$
is the \textit{minimal representative}
of its $U$-orbit in $\g/\m_i$ provided
$a_j=0$ whenever $j \leq i$ is an inert point of $X$.
It follows from the counterparts to
\cite[Prop.\ 5.4 and Lem.\ 5.5]{GOconj} in our setting
that each $U$-orbit in $\g/\m_i$
contains a unique minimal representative.
Now we can adapt the arguments in \cite[\S 2]{GMR} to show that the
minimal representatives are partitioned in to families $\mathcal{X}_c$,
where $c$ runs over the indexing set $C = \{\I,\RR_0,\RR_\nn\}^N$. For $c\in\{\I,\RR_0,\RR_\nn\}^N$, we define
$$
(\g/\b)_c=\{X+\b \in \g/\b \mid \text{$j$ is an inert point of $X$ if and only if $c_j=\I$}\}
$$
and
\begin{equation} \label{e:Xc}
\mathcal X_c=\left\{\sum_{j=1}^i a_j e_{\gamma_j}+\b \in (\g/\b)_c \:\: \vline \:\: \text{$a_j=0$
if and only if $c_j\in\{\I,\RR_0\}$}\right\}.
\end{equation}
Then the quasi-affine varieties $\mathcal X_c$ give a parametrization of the $U$-orbits in $\g/\b \cong \u^*$.

The discussion above, along with Theorem \ref{thm:sepcoad}, implies that the
algorithm from \cite[\S 3]{GMR} to calculate this parametrization in case of the adjoint action
of $U$ on $\u$ can be adapted to the coadjoint action of $U$ on $\u^*$.  This is achieved
by determining the quasi-affine varieties $\mathcal X_c$ for all $c$.  As explained in the introduction,
we have implemented this adaptation using {\sf GAP} \cite{GAP}, and used it to calculate the $U$-orbits in $\u^*$, when $G$ is
of rank at most $8$, with the exception of $E_8$.
As noted in the introduction this parametrization becomes rather intricate as the rank grows, because the
equations defining the quasi affine varieties $\mathcal X_c$ become more complicated.

\section{Counting conjugacy classes and characters in $U(q)$}
\label{sec:applic}

Now assume that $G$ is defined and split over the field $\F_q$ where
$q$ is some power of a good prime $p$ for $G$.
We denote by $G(q)$ the group of $\F_q$-rational points of $G$.
We also assume that $B$ and $T$ are chosen to be defined over $\F_q$.
Then $U$ is defined over $\F_q$ and $U(q)$ is a Sylow
$p$-subgroup of $G(q)$.

We continue to use the notation from the previous section, so we
have a fixed $B$-preorder $\preceq$ on $X(T)$ which restricts to a total order on $\Phi$, and enumerate
$\Phi^+$ and $\Phi^-$ accordingly.
Then the $\m_i$ defined in \eqref{e:m_i}
are $F$-stable $B$-submodules.
Thanks to Theorem \ref{thm:relconn}, there are analogues
of  \cite[Prop.\ 6.2 and \ Lem.\ 6.3]{GOconj} for the coadjoint action.
From this, we deduce that the
$U(q)$-orbits in $\u^*(q)$ are in bijection with
the minimal representatives with coefficients in $\F_q$.

The algorithm explained in \cite[\S3]{GMR}
can be adapted to calculate the number
$k(U(q), \u^*(q))$
of coadjoint orbits of $U(q)$ on $\u^*(q)$ from the parametrization
of the coadjoint orbits.  Our algorithm to find a parametrization of the $U$-orbits in $\u^*$
determines the quasi-affine varieties $\mathcal X_c$ as defined in \eqref{e:Xc}.  Then the number of $\F_q$-rational
points in $\mathcal X_c$ is calculated.
We have done this in {\sf GAP} and ran it for all simple groups $G$ of rank at most $8$,
except $E_8$.  We remark that the process of determining $|\mathcal X_c(q)|$
is rather involved and requires calls to {\sf SINGULAR}.
The results of our calculations are presented in
Table \ref{tab:coadjoint} below.  We note that
the results of the calculation show that $|\mathcal X_c(q)|$ is
a polynomial in $q$ for each $c$.
As we have already observed in the introduction,
the values $k(U(q),\u^*(q))$
and the number $k(U(q))$ of $U(q)$-conjugacy classes
in $U(q)$ coincide.

Table \ref{tab:coadjoint} contains the
values of $k(U(q),\u^*(q)) = K(U(q))$
for $G$ of types $E_7$, $B_8$, $C_8$ and $D_8$
as polynomials in $v = q-1$.
Combining these with our previous results
\cite[Thm.\ 1.1]{GOROconj} and
\cite[Thm.\ 1.1]{GMR}, we obtain
Theorem \ref{thm:poly}.
We recall also that for $G$ of type $A_r$, the polynomials
have been calculated up to $r = 12$ and are given in \cite{VLAR}; see also \cite{EV}.

\begin{center}
\renewcommand{\arraystretch}{1.4}
\begin{table}[h!]
\begin{tabular}{|l|l|l|}
\hline
$G(q)$ & $k(U(q),\u^*(q))$ \\
\hline
$E_7$ & $v^{17}+18v^{16}+154v^{15}+839v^{14}+3298v^{13}+10104v^{12}+25702v^{11}+57351v^{10}$ \\*
~ & $+114413v^9+194330v^8+255908v^7+238441v^6+145845v^5+54705v^4+11655v^3$ \\*
~ & $+1281v^2+63v+1$ \\
\hline
$D_8$ & $v^{16}+19v^{15}+175v^{14}+1057v^{13}+4770v^{12}+17192v^{11}+50639v^{10}+119410v^9$ \\*
~ & $+213853v^8+274244v^7+239428v^6+136164v^5+48090v^4+9912v^3+1092v^2$ \\*
~ & $+56v+1$ \\
$B_8/C_8$ & $2v^{17}+40v^{16}+387v^{15}+2422v^{14}+11077v^{13}+39613v^{12}+115125v^{11}+274653v^{10}$ \\*
~ & $+525983v^9+772250v^8+824340v^7+607950v^6+294658v^5+88816v^4+15568v^3$ \\*
~ & $+1456v^2+64v+1$ \\
\hline
\end{tabular}
\caption{$k(U(q))$
as polynomial in $v=q-1$}
\label{tab:coadjoint}
\end{table}
\end{center}

We note that the values of $k(U(q))$ are always polynomials in $v = q-1$ with non-negative integer coefficients.  This is also the case
for the polynomials obtained in \cite{GMR}.
It follows from
\cite[Thm.\ 5.4]{GMR} that
the coefficient of $k(U(q))$
(as a polynomial in $v$)
of degree $1$ equals $|\Phi^+|$ and
that of degree $2$ is also given by combinatorial
data only depending on $\Phi^+$.

It is worth noting that the algorithm is able to compute $k(U(q),\u^*(q))$
for groups for which the algorithm employed in \cite{GMR}
is not able to calculate $k(U(q),\u(q))$.
That is, the coadjoint action seems to be more
favourable for computational purposes
as opposed to the adjoint action.
A similar observation was already made by B\"urgstein and Hesselink in
\cite[\S 1.8]{BUHE} in their study of the $B$-orbits in $\u$ and $\u^*$.

As mentioned in \cite[\S 3]{GMR}, there are situations where the
algorithm might carry out implicit divisions by certain primes, and then the results
of the calculation are not necessarily valid for these primes.
Previously this did not cause any problems, as the only primes occurring were bad primes.
However, for the rank 8 cases the prime 3 showed up.  An alternative
run of the program where division by 3 was not allowed verified that the polynomials
are also valid for the prime 3.

E.~A.~O'Brien and W.~Unger have calculated the values of $k(U(q))$
for $q = p = 3$ or $5$ for $B_8$, $C_8$ and $D_8$, and
for $q = p = 5$ for $E_7$.
For this they used an improved implementation of an
algorithm by M.~C. Slattery which calculates the number of
irreducible characters of each degree for a $p$-group, \cite{SLA}.
Their calculations verified our results in these cases.

For $G$ of type $E_7$, we also used
the modified program for the coadjoint action
to calculate $k(U(q),\u^{(1)*}(q))$, where $\u^{(1)}$
is the Lie algebra of the commutator subgroup
$U^{(1)}$ of $U$. For that we used
$\u^{(1)*} \cong \g / (\b \oplus \g_{-\alpha_1} \oplus \ldots \oplus \g_{-\alpha_r})$,
where $\Pi = \{ \alpha_1, \ldots, \alpha_r \}$.
The resulting polynomial is
\begin{align*}  & v^{16}+16v^{15}+121v^{14}+579v^{13}+1983v^{12}+5232v^{11}+11268v^{10}+21028v^9 \\
 & +36019v^8+55895v^7+68476v^6+56196v^5+27571v^4+7393v^3+980v^2+56v+1.
\end{align*}
This extends the results in Table 2 from \cite{GMR}.

The sizes of coadjoint orbits $U(q)$ in $\u^*(q)$ can be determined
in analogy with \cite[Prop.\ 6.4]{GOconj}, as we state and prove below.
For shorter notation, in the statement we write $F$ for an element of $\u^*$, where we are not
making the identification $\u^* \cong \g/\b$, inert points of $F$ are defined
through the identification.

\begin{prop} \label{P:orbitsize}
Let $F \in \u^*(q)$ and let $\mathrm{in}(F)$ denote
the number of inert points of $F$.  Then $\dim U \cdot F = \mathrm{in}(F)$ and
$|U(q) \cdot F| = q^{\mathrm{in}(F)}$.
Moreover, $\mathrm{in}(F)$ is even.
\end{prop}

\begin{proof}
We can show that $\dim U \cdot F = \mathrm{in}(F)$ in the same way as the
analogous result for the adjoint action is proved in \cite[Lem.\ 5.7]{GOconj},
and $|U(q) \cdot F| = q^{\mathrm{in}(F)}$ can be proved as in \cite[Prop.\ 6.4]{GOconj}.
Since the orbit map $U \to U \cdot F$ is separable, by Theorem \ref{thm:sepcoad}, the
tangent space of $U \cdot F$ at $0$ is identified with $\u/\c_\u(F)$.  It is well-known
that there is a
symplectic form $\omega$ on $\u/\c_\u(F)$ given by $\omega(X+\c_\u(F),Y+\c_\u(F)) = F([X,Y])$.
Hence, $\dim U \cdot F$ is even.
\end{proof}

As explained in the introduction for $p \ge h$, the
Kirillov orbit method gives a bijection $\cO \mapsto \chi_\cO$ between the coadjoint
orbits of $U(q)$ in $\u^*(q)$ and the irreducible complex characters
of $U(q)$.  Moreover, the degree of $\chi_\cO$ is equal to $\sqrt{|\cO|}$
so that the degree of $\chi_\cO$ is a power of $q$ by Proposition \ref{P:orbitsize}.  We define $k(U(q),q^d)$ to be
the number of irreducible complex characters of $U(q)$ of degree $q^d$.
As mentioned above $|\mathcal X_c(q)|$ is given by a polynomial
in $q$ for all $c$.  Therefore, from the definition of $\mathcal X_c$,
we deduce that $k(U(q),q^d)$ is given by a polynomial in $q$ when $G$ has rank less than or equal to 8 except for
type $E_8$.
We remark that it is interesting to know what the situation for $p < h$ is, as discussed in the introduction.
The polynomials giving $k(U(q),q^d)$ are in the appendix of this paper; the polynomials for $G$ of type $A_r$
are not included as they can be found in \cite{Is} and \cite{EV}, though we did verify that our program gave
the same polynomial for $r \le 10$.

As noted in
\cite{GOROconj} and
\cite{GMR} for the smaller rank instances,
the polynomials in Theorem \ref{thm:poly}
coincide for types $B_r$ and $C_r$.  We note however that already for
types $B_3$ and $C_3$ the polynomials $k(U(q),q^d)$ are
not equal for all $d$, as can be observed from the table in the appendix.

\begin{rem}
\label{rem:kup}
Let $P$ be a parabolic subgroup of $G$ with unipotent radical $U_P$ and
Lie algebra $\u_P$ of $U_P$.
Without loss, we may assume that $P$ contains $B$.
It is possible to define a suitable filtration
of $\u_P$  and $\u_P^*$ so that
we obtain separability
and connectivity results analogous to
Theorems \ref{thm:sepcoad} and \ref{thm:relconn}
for the actions of $U$ on
$\u_P$ and $\u_P^*$; and we also get these for the actions of $U_P$ on
$\u_P$ and $\u_P^*$.
Then a result analogous to Lemma \ref{lem:raminco} allows one to
define \emph{inert} and \emph{ramification} points for these actions.
In turn this leads to the notion of
\emph{minimal representatives} for each of these actions.

This allows one to generalize our algorithm
to obtain a parameterization of each of:
the $U$-orbits in $\u_P$; the $U$-orbits in $\u_P^*$; the $U_P$-orbits in $\u_P$;
and the $U_P$-orbits in $\u_P^*$.

Analogous to our previous results one gets that
the number $k(U(q),U_P(q))$ of $U(q)$-conjugacy classes in $U_P(q)$
coincides with $k(U(q),\u_P(q))= k(U(q),\u_P^*(q))$, the number of
$U(q)$-conjugacy classes in $\u_P(q)$. Likewise for the action of
$U_P(q)$, we obtain that $k(U_P(q)) = k(U_P(q),\u_P(q))= k(U_P(q),\u_P^*(q))$.

To calculate the number of $U_P(q)$-classes in $U_P(q)$,
first the algorithm for $k(U(q),\u_P(q))$ is used to
determine the minimal representatives of $U(q)$-orbits in $\u_P(q)$.
Then, for each minimal representative $X\in\u_P(q)$, the number
$k(U_P(q),U(q)\cdot X)$ of all $U_P(q)$-orbits in the $U(q)$-orbit of $X$
is calculated using the formula given in \cite[Lem.\ 5.2.1]{mosch:phd}.

We refer to \cite[Ch.\ 5]{mosch:phd} for further details of these parabolic
analogues and for a list of explicit results obtained using
the parabolic version of the algorithm.

Similar results regarding the irreducible characters of $U_P(q)$ can
also be obtained.
\end{rem}

\section{The modality of the action of $B$ on $\u$}
\label{subsect:modality}

In this section, we record another application of
the calculations to parameterize the $U$-orbits in $\u$ and $\u^*$.
Namely to determine the modality of the action $B$ in $\u$
in the cases we have considered in
Section \ref{sec:applic}.
This is  particularly meaningful since, at present, there appears
to be no effective method
to calculate a good upper bound for $\modd(B : \u)$.
Our results allow us to extend the
list of known values of $\modd(B : \u)$, which are given in \cite[Tables II and III]{jurgensroehrle}.
As a general reference for the modality of the action of an algebraic group
on an algebraic variety, we refer the reader to \cite{Vi} or \cite[Sec.\ 5.2]{VP}.

Recall that the modality of the action of the
Borel subgroup $B$ on $\u$ is defined by
\[
\modd(B : \u) := \max_{i \in \Z_{\ge 0}} ( \dim \u_i - i),
\]
where $\u_i := \{ X \in \u \mid \dim B \cdot X = i\}$.
Likewise, the modality of the action of $U$ on $\u$
is defined analogously by replacing the adjoint $B$-action on $\u$ by
the $U$-action, i.e.,
\[
\modd(U : \u) := \max_{i \in \Z_{\ge 0}} ( \dim \{ X \in \u \mid \dim U \cdot X = i\} - i).
\]

Each $U$-orbit in $\u$ admits a minimal representative, and
as explained in \cite[\S 2]{GOconj}, the minimal
representatives are partitioned into families $\mathcal{X}_c$,
where $c$ runs over $C = \{\I,\RR_0,\RR_\nn\}^N$.  Then $\modd(U:\u)$ is just the
largest dimension of a family of minimal representatives.
We also note that $\modd(U:\u) = \modd(U:\u^*)$, \cite[Thm.\ 1.4]{roehrle:modality2},
so we can determine this modality by considering the coadjoint action of $U$ on $\u^*$.

In the following theorem, we
show that $\modd(B : \u)$ can easily be determined from $\modd(U : \u)$; the parametrization by minimal representatives
is required for the proof.

\begin{thm}
\label{thm:modb}
Let $G$ be a simple algebraic group and suppose that $p$ is
either zero or a good prime for $G$.
Let $B$ be a Borel subgroup of $G$ with unipotent radical $U$.
Then
\[
\modd(U : \u)  - \rank G = \modd(B : \u).
\]
\end{thm}

\begin{proof}
By \cite[Prop.\ 7.7]{GOconj},
we see that for any minimal representative $X \in \u$
we have
\begin{equation} \label{e:centprod}
C_B(X) = C_T(X) C_U(X)
\end{equation}
and we note that the definition of the families $\mathcal{X}_c$ ensures that
$\dim T \cdot X$ is independent of the choice of $X \in \mathcal{X}_c$.
Each family $\mathcal{X}_c$ of $U$-orbits in $\u$ gives rise to
a family of $B$-orbits of dimension
\begin{equation}
\label{eq:Xc}
\dim \mathcal{X}_c - \dim T \cdot X =
\dim \mathcal{X}_c - \rank G + \dim C_T(X),
\end{equation}
where $X \in \mathcal X_c$.

Let $\u_{\dist}$ be the union of all $U$-orbits in $\u$
such that $\dim C_T(X) = 0$ for $X$ the minimal representative of the orbit;
we note that this is equivalent to
$C_T(X)^\circ = 1$.
Note that thanks to \eqref{e:centprod},  we have that $\u_{\dist}$
is the $B$-stable subvariety of $\u$ consisting of all
$X \in \u$ such that $C_B(X)^\circ$ is unipotent.
Then by \eqref{eq:Xc}, we have
$
\modd(B:\u_{\dist}) = \modd(U:\u_{\dist}) - \rank G.
$

To complete the proof, we show by induction on the rank of $G$ that
\begin{equation} \label{e:modal1}
\modd(B:\u) = \modd(B:\u_{\dist})
\end{equation}
and
\begin{equation} \label{e:modal2}
\modd(U:\u) = \modd(U:\u_{\dist}).
\end{equation}
The base case where $G$ has rank 0 is trivial, so we assume inductively that
\eqref{e:modal1} and \eqref{e:modal2} are true in lower rank.

Let $\u_{\dec} := \u \setminus \u_{\dist}$ be the complement of $\u_{\dist}$ in $\u$.
We show that
$$\modd(U:\u_{\dec}) \le \modd(U:\u_{\dist}) \quad \text{and} \quad
\modd(B:\u_{\dec}) \le \modd(B:\u_{\dist}),
$$
which gives the inductive step.

Let $X \in \u_{\dec}$.  Then
there is a torus $S \subseteq T$ of positive dimension
contained in $C_B(X)$.  We have that $L=C_G(S)$ is a proper Levi subgroup of $G$, and
moreover, $L \cap B$ is a Borel subgroup of $L$
and $\l \cap \u$ is the Lie algebra of the unipotent radical of $L \cap B$.
Then the intersection of a $B$-orbit in $\u$ with $\l \cap \u$ is a union of $(L \cap B)$-orbits; and also the intersection
of a $U$-orbit in $\u$ with $\l \cap \u$ is a union of $(L \cap U)$-orbits.
Also, up to conjugacy by $B$,
there are only finitely many possibilities for $L$,
(each such is determined by
its root system which is a subsystem of the root system of $G$).

Therefore, we have that
\[
\modd(B:\u_\dec) \le \max_L \modd(L \cap B : \l \cap \u ),
\]
where the maximum is taken over all Levi subgroups $L$ of $G$ containing $T$.
Similarly,
\[
\modd(U:\u_\dec) \le \max_L \modd(L \cap U : \l \cap \u ).
\]

We proceed to show that
$$
\modd(L \cap B : \l \cap \u ) \le \modd(B:\u_{\dist})
$$
and that
$$
\modd(L \cap U : \l \cap \u) \le \modd(U:\u_{\dist}),
$$
when $L$ is a Levi subgroup of $G$ containing $T$.
For this we may assume that the simple roots of $L$ determined by
$L \cap B$ are a subset of the simple roots of $G$ determined by $B$.

By the inductive hypothesis, there is a family $\mathcal{X}_c(\l)$ of minimal
representatives of $(L \cap U)$-orbits in $\l \cap \u$ such that:
\begin{enumerate}
\item[(i)] $C_T(X)^\circ = Z(L)^\circ$ for all $X \in \mathcal{X}_c(\l)$;
\item[(ii)] $\dim \mathcal{X}_c(\l) = \modd(L \cap U: \l \cap \u)$; and
\item[(iii)] $\dim \mathcal{X}_c(\l) - \rank G  + \dim Z(L) = \modd(L \cap B: \l \cap \u)$.
\end{enumerate}
Let $\alpha_1,\dots,\alpha_h$ be the simple roots of $G$ that are not simple roots of $L$.
We consider
$$
\mathcal Y_c = \left\{\sum_{i=1}^h a_i e_{\alpha_i} + X \:\: \vline \:\:  a_i \in k^\times \text{ and } X \in \mathcal{X}_c(\l)\right\}.
$$
Clearly, the elements of $\mathcal Y_c$ lie in distinct $U$-orbits in $\u$.
Moreover,
$\mathcal Y_c \subseteq \u_\dist$, and $\dim \mathcal Y_c = \dim X_c(\l) + h$, so we deduce
that $\modd(L \cap U : \l \cap \u) \le \modd(U:\u_{\dist})$.

We have that $T$ acts on $\mathcal Y_c$ and the
dimension of each $T$-orbit is $\dim T$.  Therefore,
$\mathcal Y_c$ gives a family of $B$-orbits in $\u_\dist$ of dimension $\dim \mathcal{X}_c(\l) - \rank G  + \dim Z(L)$.
From this it follows that $\modd(L \cap B: \l \cap \u) \le \modd(B:\u_{\dist})$.
This completes the proof.
\end{proof}

In the cases in which we have calculated $k(U(q))$, the modality
$\modd(U:\u)$ is simply the degree of the polynomial giving $k(U(q))$.
Then we can calculate $\modd(B:\u)$ using Theorem \ref{thm:modb}.
It is also possible to determine $\modd(B:\u^*)$ directly from our parametrization of the $U$-orbits
in $\u^*$ and use that $\modd(B:\u) = \modd(B:\u^*)$, by \cite[Thm.\ 1.4]{roehrle:modality2}.

This allows us to
extend the list of values for $\modd(B : \u)$ as follows.
The list of previously known values for $\modd(B : \u)$
for smaller rank groups is given in
\cite[Tables II and III]{jurgensroehrle}.
Below for completeness we also include the values of this modality
for $G$ of type $A_n$, which we obtain via the polynomials calculated in
\cite[\S 5]{VLAR}.

\begin{table}[ht!b]
\renewcommand{\arraystretch}{1.5}
\begin{tabular}{r|cccccccc}\hline
Type of $G$ &  $A_{11}$ & $A_{12}$ & $B_7$ & $B_8$ & $C_8$ & $D_8$ & $E_7$ \\
\hline
$\modd(B : \u)$ & 7 & 8 & 7 & 9 & 9 & 8 & 10  \\
\hline
\end{tabular}
%\medskip
\caption{Modality of the action of $B$ on $\u$}
\label{modality-table}
\end{table}

Note that the values given in
Table \ref{modality-table} were previously only known to be
lower bounds for $\modd(B : \u)$ and were obtained by means of a rather
coarse dimension estimate; see \cite[Prop.\ 3.3, Table 4]{roehrle:modality}.
The construction of the lower bounds
for $\modd(B : \u)$
is given in the following way.
There exists a normal subgroup $A$ of $B$ contained in $U$ with Lie algebra
$\a$, such that
$\dim \a/[\a,\a] - \dim B/A$ is quite large, and this is the lower bound given.
To explain this further, we note that the adjoint action of
$B$ on $\u$ restricts to an action on $\a$, which in turn
induces an action of $B/A$ on $\a/[\a,\a]$.  Then clearly,
we have that $\modd(B:\u) \ge \dim \a/[\a,\a] - \dim B/A$.
Our results show that in fact in the cases where the modality is
known, it is actually obtained in this way.

\appendix

\section{The polynomials $k(U(q),q^d)$}

In this appendix we include the polynomials
$k(U(q),q^d)$, which give the number of irreducible characters of $U(q)$
of degree $q^d$.  These polynomials are given as polynomials in
$v = q-1$ and are in the table below.  We note that all the coefficients
are positive, which verifies that the analogue of Isaacs' conjecture holds in
these cases.

We do not list the polynomials for $G$
of type $A$ as these are included in \cite{Is} and \cite{EV}; we have
verified that our program gives the same polynomials for $A_r$ for $r \le 10$.

We remark that $k(U(q),q^0)$ is always equal to $q^r$, where $r$ is the rank of $G$, but
we include these polynomials for completeness.

%\newpage

\begin{center} \tiny
\renewcommand{\arraystretch}{1}
\begin{longtable}{|l|l|p{410pt}|}
\hline
$G(q)$ & $d$ & $k(U(q),q^d)$ \\
\hline
\endhead
\hline
\endfoot
\hline
\caption{$k(U(q),q^d)$ as polynomial in $v=q-1$}
\endlastfoot
$B_2$ & 0 & $v^2+2v+1$ \\
 & 1 & $v^2+2v$ \\
\hline
$G_2$ & 0 & $v^2+2v+1$ \\
 & 1 & $v^3+3v^2+3v$ \\
 & 2 & $v^2+v$ \\
\hline
$B_3$ & 0 & $v^3+3v^2+3v+1$ \\
 & 1 & $2v^3+5v^2+3v$ \\
 & 2 & $v^4+4v^3+6v^2+2v$ \\
 & 3 & $v^3+2v^2+v$ \\
\hline
$C_3$ & 0 & $v^3+3v^2+3v+1$ \\
 & 1 & $v^4+4v^3+6v^2+3v$ \\
 & 2 & $2v^3+4v^2+2v$ \\
 & 3 & $v^3+3v^2+v$ \\
\hline
$B_4$ & 0 & $v^4+4v^3+6v^2+4v+1$ \\
 & 1 & $3v^4+10v^3+11v^2+4v$ \\
 & 2 & $v^6+6v^5+16v^4+23v^3+15v^2+3v$ \\
 & 3 & $2v^5+10v^4+18v^3+13v^2+3v$ \\
 & 4 & $2v^5+11v^4+19v^3+11v^2+v$ \\
 & 5 & $v^5+6v^4+11v^3+6v^2+v$ \\
 & 6 & $v^4+3v^3+2v^2$ \\
\hline
$C_4$ & 0 & $v^4+4v^3+6v^2+4v+1$ \\
 & 1 & $v^5+6v^4+13v^3+12v^2+4v$ \\
 & 2 & $2v^5+11v^4+20v^3+14v^2+3v$ \\
 & 3 & $v^6+6v^5+17v^4+24v^3+14v^2+3v$ \\
 & 4 & $2v^5+9v^4+15v^3+9v^2+v$ \\
 & 5 & $3v^4+8v^3+6v^2+v$ \\
 & 6 & $v^4+4v^3+3v^2$ \\
\hline
$D_4$ & 0 & $v^4+4v^3+6v^2+4v+1$ \\
 & 1 & $v^5+5v^4+10v^3+9v^2+3v$ \\
 & 2 & $3v^4+9v^3+9v^2+3v$ \\
 & 3 & $v^5+5v^4+10v^3+7v^2+v$ \\
 & 4 & $v^4+3v^3+3v^2+v$ \\
\hline
$F_4$ & 0 & $v^4+4v^3+6v^2+4v+1$ \\
 & 1 & $v^5+6v^4+13v^3+12v^2+4v$ \\
 & 2 & $v^6+7v^5+20v^4+28v^3+18v^2+4v$ \\
 & 3 & $4v^5+20v^4+33v^3+21v^2+4v$ \\
 & 4 & $v^8+8v^7+28v^6+58v^5+79v^4+66v^3+24v^2+2v$ \\
 & 5 & $v^7+7v^6+22v^5+39v^4+37v^3+15v^2+2v$ \\
 & 6 & $2v^6+14v^5+36v^4+40v^3+17v^2+2v$ \\
 & 7 & $2v^6+13v^5+32v^4+34v^3+13v^2+2v$ \\
 & 8 & $4v^5+15v^4+19v^3+8v^2$ \\
 & 9 & $v^5+7v^4+11v^3+5v^2$ \\
 & 10 & $v^4+3v^3+v^2$ \\
\hline
$B_5$ & 0 & $v^5+5v^4+10v^3+10v^2+5v+1$ \\
 & 1 & $4v^5+17v^4+27v^3+19v^2+5v$ \\
 & 2 & $v^7+8v^6+29v^5+56v^4+56v^3+26v^2+4v$ \\
 & 3 & $2v^7+15v^6+49v^5+83v^4+72v^3+29v^2+4v$ \\
 & 4 & $v^8+9v^7+37v^6+87v^5+119v^4+87v^3+29v^2+3v$ \\
 & 5 & $3v^7+23v^6+70v^5+105v^4+77v^3+24v^2+2v$ \\
 & 6 & $9v^6+48v^5+90v^4+71v^3+21v^2+v$ \\
 & 7 & $v^8+8v^7+31v^6+73v^5+96v^4+57v^3+11v^2+v$ \\
 & 8 & $v^7+7v^6+23v^5+37v^4+26v^3+7v^2$ \\
 & 9 & $2v^6+10v^5+19v^4+14v^3+3v^2$ \\
 & 10 & $v^5+3v^4+3v^3+v^2$ \\
\hline
$C_5$ & 0 & $v^5+5v^4+10v^3+10v^2+5v+1$ \\
 & 1 & $v^6+8v^5+23v^4+31v^3+20v^2+5v$ \\
 & 2 & $v^7+8v^6+29v^5+55v^4+54v^3+25v^2+4v$ \\
 & 3 & $v^7+12v^6+48v^5+88v^4+78v^3+31v^2+4v$ \\
 & 4 & $v^8+9v^7+38v^6+91v^5+123v^4+86v^3+27v^2+3v$ \\
 & 5 & $3v^7+24v^6+74v^5+110v^4+80v^3+25v^2+2v$ \\
 & 6 & $v^8+8v^7+32v^6+78v^5+105v^4+69v^3+18v^2+v$ \\
 & 7 & $2v^7+14v^6+44v^5+68v^4+46v^3+12v^2+v$ \\
 & 8 & $3v^6+17v^5+34v^4+26v^3+7v^2$ \\
 & 9 & $4v^5+14v^4+14v^3+4v^2$ \\
 & 10 & $v^5+5v^4+6v^3+v^2$ \\
\hline
$D_5$ & 0 & $v^5+5v^4+10v^3+10v^2+5v+1$ \\
 & 1 & $v^6+7v^5+19v^4+25v^3+16v^2+4v$ \\
 & 2 & $2v^6+14v^5+35v^4+40v^3+21v^2+4v$ \\
 & 3 & $v^7+8v^6+29v^5+54v^4+50v^3+21v^2+3v$ \\
 & 4 & $2v^6+17v^5+42v^4+42v^3+17v^2+2v$ \\
 & 5 & $v^7+8v^6+29v^5+53v^4+43v^3+14v^2+v$ \\
 & 6 & $v^6+7v^5+18v^4+18v^3+7v^2+v$ \\
 & 7 & $2v^5+8v^4+10v^3+3v^2$ \\
 & 8 & $v^4+2v^3+v^2$ \\
\hline
$B_6$ & 0 & $v^6+6v^5+15v^4+20v^3+15v^2+6v+1$ \\
 & 1 & $5v^6+26v^5+54v^4+56v^3+29v^2+6v$ \\
 & 2 & $v^8+10v^7+46v^6+113v^5+153v^4+112v^3+40v^2+5v$ \\
 & 3 & $v^9+10v^8+47v^7+135v^6+243v^5+262v^4+158v^3+47v^2+5v$ \\
 & 4 & $v^9+15v^8+84v^7+248v^6+421v^5+410v^4+216v^3+53v^2+4v$ \\
 & 5 & $v^{10}+11v^9+61v^8+209v^7+457v^6+628v^5+516v^4+234v^3+51v^2+4v$ \\
 & 6 & $4v^9+43v^8+198v^7+495v^6+713v^5+585v^4+254v^3+48v^2+2v$ \\
 & 7 & $v^{11}+11v^{10}+56v^9+183v^8+438v^7+765v^6+896v^5+632v^4+237v^3+39v^2+2v$ \\
 & 8 & $2v^{10}+20v^9+95v^8+290v^7+589v^6+751v^5+554v^4+213v^3+35v^2+v$ \\
 & 9 & $4v^9+39v^8+178v^7+450v^6+634v^5+477v^4+171v^3+22v^2+v$ \\
 & 10 & $v^{10}+12v^9+67v^8+224v^7+468v^6+582v^5+392v^4+124v^3+14v^2$ \\
 & 11 & $2v^9+21v^8+100v^7+259v^6+365v^5+264v^4+86v^3+9v^2$ \\
 & 12 & $v^9+11v^8+54v^7+144v^6+205v^5+142v^4+43v^3+5v^2$ \\
 & 13 & $v^8+11v^7+45v^6+84v^5+71v^4+25v^3+2v^2$ \\
 & 14 & $2v^7+13v^6+30v^5+28v^4+9v^3+v^2$ \\
 & 15 & $v^6+4v^5+5v^4+2v^3$ \\
\hline
$C_6$ & 0 & $v^6+6v^5+15v^4+20v^3+15v^2+6v+1$ \\
 & 1 & $v^7+10v^6+36v^5+64v^4+61v^3+30v^2+6v$ \\
 & 2 & $v^8+11v^7+50v^6+118v^5+154v^4+110v^3+39v^2+5v$ \\
 & 3 & $4v^8+35v^7+130v^6+256v^5+281v^4+168v^3+49v^2+5v$ \\
 & 4 & $v^{10}+10v^9+47v^8+146v^7+321v^6+472v^5+426v^4+215v^3+52v^2+4v$ \\
 & 5 & $3v^9+31v^8+148v^7+392v^6+599v^5+520v^4+241v^3+52v^2+4v$ \\
 & 6 & $5v^9+51v^8+230v^7+565v^6+790v^5+620v^4+255v^3+46v^2+2v$ \\
 & 7 & $v^{11}+11v^{10}+56v^9+186v^8+455v^7+803v^6+932v^5+644v^4+238v^3+40v^2+2v$ \\
 & 8 & $2v^{10}+21v^9+104v^8+318v^7+628v^6+769v^5+542v^4+199v^3+31v^2+v$ \\
 & 9 & $5v^9+49v^8+208v^7+483v^6+638v^5+460v^4+163v^3+23v^2+v$ \\
 & 10 & $v^{10}+10v^9+52v^8+178v^7+391v^6+510v^5+361v^4+122v^3+15v^2$ \\
 & 11 & $2v^9+19v^8+85v^7+220v^6+318v^5+235v^4+78v^3+9v^2$ \\
 & 12 & $3v^8+26v^7+93v^6+162v^5+133v^4+47v^3+6v^2$ \\
 & 13 & $4v^7+28v^6+67v^5+66v^4+26v^3+2v^2$ \\
 & 14 & $5v^6+22v^5+29v^4+13v^3+v^2$ \\
 & 15 & $v^6+6v^5+10v^4+4v^3$ \\
\hline
$D_6$ & 0 & $v^6+6v^5+15v^4+20v^3+15v^2+6v+1$ \\
 & 1 & $v^7+9v^6+31v^5+54v^4+51v^3+25v^2+5v$ \\
 & 2 & $v^8+9v^7+38v^6+89v^5+119v^4+89v^3+34v^2+5v$ \\
 & 3 & $v^8+15v^7+72v^6+165v^5+201v^4+130v^3+40v^2+4v$ \\
 & 4 & $3v^8+31v^7+124v^6+246v^5+260v^4+145v^3+39v^2+4v$ \\
 & 5 & $v^{10}+10v^9+46v^8+135v^7+280v^6+393v^5+339v^4+163v^3+36v^2+2v$ \\
 & 6 & $2v^9+18v^8+77v^7+200v^6+317v^5+288v^4+138v^3+30v^2+2v$ \\
 & 7 & $5v^8+43v^7+154v^6+282v^5+270v^4+128v^3+25v^2+v$ \\
 & 8 & $3v^8+31v^7+122v^6+227v^5+208v^4+89v^3+15v^2+v$ \\
 & 9 & $v^9+9v^8+41v^7+113v^6+181v^5+152v^4+61v^3+8v^2$ \\
 & 10 & $v^8+8v^7+31v^6+62v^5+61v^4+27v^3+5v^2$ \\
 & 11 & $2v^7+12v^6+29v^5+32v^4+15v^3+2v^2$ \\
 & 12 & $v^6+4v^5+6v^4+4v^3+v^2$ \\
\hline
$E_6$ & 0 & $v^6+6v^5+15v^4+20v^3+15v^2+6v+1$ \\
 & 1 & $v^7+9v^6+31v^5+54v^4+51v^3+25v^2+5v$ \\
 & 2 & $5v^7+34v^6+93v^5+130v^4+97v^3+36v^2+5v$ \\
 & 3 & $v^9+9v^8+42v^7+123v^6+223v^5+240v^4+145v^3+44v^2+5v$ \\
 & 4 & $5v^8+42v^7+155v^6+300v^5+316v^4+176v^3+46v^2+4v$ \\
 & 5 & $2v^9+23v^8+118v^7+327v^6+518v^5+462v^4+219v^3+48v^2+3v$ \\
 & 6 & $14v^8+113v^7+367v^6+602v^5+523v^4+231v^3+45v^2+3v$ \\
 & 7 & $v^{11}+11v^{10}+57v^9+186v^8+433v^7+730v^6+826v^5+560v^4+204v^3+36v^2+2v$ \\
 & 8 & $v^{10}+10v^9+51v^8+173v^7+396v^6+558v^5+444v^4+183v^3+31v^2+v$ \\
 & 9 & $3v^9+30v^8+144v^7+385v^6+575v^5+455v^4+177v^3+28v^2+v$ \\
 & 10 & $12v^8+95v^7+304v^6+480v^5+375v^4+131v^3+16v^2+v$ \\
 & 11 & $2v^9+21v^8+97v^7+243v^6+334v^5+233v^4+71v^3+10v^2$ \\
 & 12 & $2v^8+20v^7+76v^6+139v^5+124v^4+49v^3+6v^2$ \\
 & 13 & $3v^7+24v^6+63v^5+68v^4+28v^3+3v^2$ \\
 & 14 & $4v^6+19v^5+27v^4+12v^3+v^2$ \\
 & 15 & $3v^5+8v^4+5v^3$ \\
 & 16 & $v^4+v^3$ \\
\hline
$B_7$ & 0 & $v^7+7v^6+21v^5+35v^4+35v^3+21v^2+7v+1$ \\
 & 1 & $6v^7+37v^6+95v^5+130v^4+100v^3+41v^2+7v$ \\
 & 2 & $v^9+12v^8+67v^7+201v^6+345v^5+346v^4+197v^3+57v^2+6v$ \\
 & 3 & $v^{10}+13v^9+76v^8+262v^7+564v^6+757v^5+618v^4+290v^3+69v^2+6v$ \\
 & 4 & $4v^{10}+44v^9+222v^8+660v^7+1228v^6+1431v^5+1013v^4+406v^3+79v^2+5v$ \\
 & 5 & $v^{12}+12v^{11}+69v^{10}+264v^9+747v^8+1565v^7+2316v^6+2280v^5+1400v^4+492v^3+85v^2+5v$ \\
 & 6 & $3v^{11}+39v^{10}+244v^9+910v^8+2147v^7+3246v^6+3095v^5+1783v^4+570v^3+85v^2+4v$ \\
 & 7 & $v^{12}+17v^{11}+131v^{10}+610v^9+1859v^8+3762v^7+4995v^6+4218v^5+2145v^4+598v^3+77v^2+3v$ \\
 & 8 & $v^{13}+15v^{12}+103v^{11}+445v^{10}+1382v^9+3205v^8+5420v^7+6366v^6+4909v^5+2326v^4+615v^3+74v^2+2v$ \\
 & 9 & $3v^{12}+41v^{11}+266v^{10}+1081v^9+2961v^8+5474v^7+6667v^6+5145v^5+2363v^4+578v^3+61v^2+2v$ \\
 & 10 & $v^{14}+14v^{13}+92v^{12}+383v^{11}+1154v^{10}+2717v^9+5129v^8+7494v^7+7871v^6+5474v^5+2311v^4+529v^3+54v^2+v$ \\
 & 11 & $2v^{13}+27v^{12}+172v^{11}+692v^{10}+1983v^9+4235v^8+6635v^7+7220v^6+5088v^5+2131v^4+460v^3+37v^2+v$ \\
 & 12 & $5v^{12}+62v^{11}+361v^{10}+1310v^9+3245v^8+5530v^7+6254v^6+4422v^5+1793v^4+360v^3+27v^2$ \\
 & 13 & $9v^{11}+111v^{10}+618v^9+2007v^8+4033v^7+4997v^6+3678v^5+1491v^4+286v^3+18v^2$ \\
 & 14 & $5v^{11}+74v^{10}+464v^9+1618v^8+3362v^7+4166v^6+2965v^5+1118v^4+193v^3+13v^2$ \\
 & 15 & $v^{13}+13v^{12}+78v^{11}+299v^{10}+847v^9+1862v^8+3039v^7+3343v^6+2238v^5+821v^4+142v^3+7v^2$ \\
 & 16 & $v^{12}+12v^{11}+69v^{10}+264v^9+736v^8+1439v^7+1800v^6+1307v^5+492v^4+79v^3+4v^2$ \\
 & 17 & $2v^{11}+22v^{10}+112v^9+355v^8+745v^7+983v^6+742v^5+289v^4+50v^3+2v^2$ \\
 & 18 & $3v^{10}+28v^9+122v^8+311v^7+469v^6+389v^5+158v^4+24v^3+v^2$ \\
 & 19 & $3v^9+24v^8+84v^7+153v^6+144v^5+64v^4+11v^3$ \\
 & 20 & $3v^8+18v^7+43v^6+48v^5+24v^4+4v^3$ \\
 & 21 & $v^7+4v^6+6v^5+4v^4+v^3$ \\
\hline
$C_7$ & 0 & $v^7+7v^6+21v^5+35v^4+35v^3+21v^2+7v+1$ \\
 & 1 & $v^8+12v^7+52v^6+115v^5+145v^4+106v^3+42v^2+7v$ \\
 & 2 & $v^9+14v^8+77v^7+220v^6+361v^5+350v^4+195v^3+56v^2+6v$ \\
 & 3 & $v^{10}+13v^9+79v^8+280v^7+606v^6+807v^5+651v^4+302v^3+71v^2+6v$ \\
 & 4 & $v^{11}+11v^{10}+67v^9+275v^8+756v^7+1351v^6+1526v^5+1049v^4+409v^3+78v^2+5v$ \\
 & 5 & $2v^{11}+27v^{10}+167v^9+623v^8+1501v^7+2350v^6+2345v^5+1435v^4+502v^3+87v^2+5v$ \\
 & 6 & $v^{12}+14v^{11}+95v^{10}+412v^9+1243v^8+2605v^7+3672v^6+3341v^5+1858v^4+576v^3+83v^2+4v$ \\
 & 7 & $v^{13}+13v^{12}+82v^{11}+337v^{10}+1021v^9+2389v^8+4209v^7+5237v^6+4295v^5+2159v^4+603v^3+79v^2+3v$ \\
 & 8 & $3v^{12}+36v^{11}+220v^{10}+895v^9+2539v^8+4908v^7+6231v^6+4985v^5+2374v^4+611v^3+70v^2+2v$ \\
 & 9 & $v^{13}+15v^{12}+109v^{11}+502v^{10}+1632v^9+3859v^8+6488v^7+7401v^6+5425v^5+2393v^4+576v^3+63v^2+2v$ \\
 & 10 & $4v^{12}+55v^{11}+358v^{10}+1429v^9+3764v^8+6605v^7+7570v^6+5446v^5+2306v^4+515v^3+49v^2+v$ \\
 & 11 & $v^{14}+14v^{13}+91v^{12}+372v^{11}+1104v^{10}+2602v^9+4983v^8+7346v^7+7643v^6+5157v^5+2067v^4+435v^3+38v^2+v$ \\
 & 12 & $2v^{13}+26v^{12}+157v^{11}+606v^{10}+1717v^9+3715v^8+5899v^7+6388v^6+4375v^5+1744v^4+357v^3+28v^2$ \\
 & 13 & $4v^{12}+49v^{11}+282v^{10}+1021v^9+2551v^8+4395v^7+4985v^6+3491v^5+1388v^4+272v^3+19v^2$ \\
 & 14 & $8v^{11}+93v^{10}+491v^9+1533v^8+3015v^7+3695v^6+2692v^5+1075v^4+201v^3+13v^2$ \\
 & 15 & $v^{12}+12v^{11}+77v^{10}+338v^9+1030v^8+2084v^7+2639v^6+1954v^5+771v^4+137v^3+8v^2$ \\
 & 16 & $2v^{11}+24v^{10}+140v^9+506v^8+1155v^7+1586v^6+1223v^5+485v^4+85v^3+4v^2$ \\
 & 17 & $3v^{10}+35v^9+177v^8+498v^7+792v^6+675v^5+287v^4+52v^3+2v^2$ \\
 & 18 & $4v^9+42v^8+172v^7+346v^6+346v^5+161v^4+30v^3+v^2$ \\
 & 19 & $5v^8+42v^7+120v^6+148v^5+78v^4+14v^3$ \\
 & 20 & $6v^7+32v^6+54v^5+34v^4+6v^3$ \\
 & 21 & $v^7+7v^6+15v^5+10v^4+v^3$ \\
\hline
$D_7$ & 0 & $v^7+7v^6+21v^5+35v^4+35v^3+21v^2+7v+1$ \\
 & 1 & $v^8+11v^7+46v^6+100v^5+125v^4+91v^3+36v^2+6v$ \\
 & 2 & $v^9+12v^8+62v^7+174v^6+287v^5+284v^4+164v^3+50v^2+6v$ \\
 & 3 & $4v^9+41v^8+180v^7+432v^6+611v^5+515v^4+248v^3+60v^2+5v$ \\
 & 4 & $2v^{10}+22v^9+123v^8+409v^7+828v^6+1023v^5+758v^4+319v^3+67v^2+5v$ \\
 & 5 & $v^{11}+11v^{10}+66v^9+281v^8+801v^7+1439v^6+1585v^5+1039v^4+380v^3+67v^2+4v$ \\
 & 6 & $v^{12}+13v^{11}+78v^{10}+296v^9+808v^8+1621v^7+2280v^6+2110v^5+1208v^4+394v^3+62v^2+3v$ \\
 & 7 & $3v^{11}+36v^{10}+203v^9+715v^8+1660v^7+2501v^6+2366v^5+1339v^4+415v^3+58v^2+2v$ \\
 & 8 & $v^{12}+13v^{11}+85v^{10}+358v^9+1050v^8+2146v^7+2933v^6+2541v^5+1313v^4+370v^3+48v^2+2v$ \\
 & 9 & $v^{12}+14v^{11}+91v^{10}+373v^9+1066v^8+2140v^7+2890v^6+2476v^5+1261v^4+345v^3+41v^2+v$ \\
 & 10 & $2v^{11}+28v^{10}+181v^9+694v^8+1653v^7+2430v^6+2136v^5+1063v^4+269v^3+27v^2+v$ \\
 & 11 & $4v^{10}+65v^9+391v^8+1177v^7+1940v^6+1779v^5+876v^4+208v^3+18v^2$ \\
 & 12 & $v^{12}+12v^{11}+68v^{10}+251v^9+686v^8+1367v^7+1805v^6+1428v^5+620v^4+136v^3+12v^2$ \\
 & 13 & $v^{11}+12v^{10}+69v^9+257v^8+639v^7+990v^6+891v^5+438v^4+100v^3+7v^2$ \\
 & 14 & $2v^{10}+21v^9+101v^8+290v^7+496v^6+469v^5+226v^4+49v^3+4v^2$ \\
 & 15 & $3v^9+27v^8+107v^7+226v^6+249v^5+132v^4+30v^3+2v^2$ \\
 & 16 & $3v^8+22v^7+64v^6+86v^5+51v^4+12v^3+v^2$ \\
 & 17 & $3v^7+15v^6+27v^5+19v^4+4v^3$ \\
 & 18 & $v^6+3v^5+3v^4+v^3$ \\
\hline
$E_7$ & 0 & $v^7+7v^6+21v^5+35v^4+35v^3+21v^2+7v+1$ \\
 & 1 & $v^8+11v^7+46v^6+100v^5+125v^4+91v^3+36v^2+6v$ \\
 & 2 & $v^9+13v^8+69v^7+194v^6+317v^5+309v^4+175v^3+52v^2+6v$ \\
 & 3 & $v^{10}+12v^9+73v^8+260v^7+562v^6+747v^5+604v^4+283v^3+68v^2+6v$ \\
 & 4 & $3v^{10}+35v^9+191v^8+601v^7+1149v^6+1350v^5+959v^4+390v^3+80v^2+6v$ \\
 & 5 & $v^{12}+12v^{11}+69v^{10}+267v^9+777v^8+1669v^7+2491v^6+2442v^5+1489v^4+522v^3+90v^2+5v$ \\
 & 6 & $3v^{11}+45v^{10}+301v^9+1145v^8+2676v^7+3936v^6+3623v^5+2018v^4+631v^3+95v^2+5v$ \\
 & 7 & $v^{14}+14v^{13}+91v^{12}+368v^{11}+1058v^{10}+2367v^9+4336v^8+6398v^7+7081v^6+5404v^5+2621v^4+733v^3+100v^2+4v$ \\
 & 8 & $2v^{13}+27v^{12}+170v^{11}+684v^{10}+2009v^9+4482v^8+7390v^7+8520v^6+6470v^5+3027v^4+793v^3+98v^2+4v$ \\
 & 9 & $v^{13}+18v^{12}+143v^{11}+696v^{10}+2360v^9+5736v^8+9758v^7+11153v^6+8175v^5+3621v^4+882v^3+98v^2+3v$ \\
 & 10 & $v^{14}+14v^{13}+99v^{12}+476v^{11}+1723v^{10}+4811v^9+10110v^8+15235v^7+15622v^6+10320v^5+4115v^4+898v^3+90v^2+3v$ \\
 & 11 & $v^{14}+15v^{13}+109v^{12}+531v^{11}+1968v^{10}+5640v^9+11968v^8+17834v^7+17810v^6+11382v^5+4380v^4+919v^3+86v^2+2v$ \\
 & 12 & $6v^{13}+88v^{12}+622v^{11}+2747v^{10}+8198v^9+16835v^8+23543v^7+21810v^6+12816v^5+4475v^4+839v^3+72v^2+2v$ \\
 & 13 & $v^{17}+17v^{16}+136v^{15}+680v^{14}+2382v^{13}+6226v^{12}+12696v^{11}+21033v^{10}+29378v^9+35038v^8+34367v^7+25583v^6+13175v^5+4287v^4+786v^3+64v^2+v$ \\
 & 14 & $v^{16}+16v^{15}+120v^{14}+561v^{13}+1840v^{12}+4555v^{11}+9054v^{10}+15182v^9$ \\*
 & & $+21573v^8+24525v^7+20430v^6+11437v^5+3934v^4+733v^3+59v^2+v$ \\
 & 15 & $2v^{15}+30v^{14}+210v^{13}+919v^{12}+2859v^{11}+6841v^{10}+13176v^9+20433v^8$ \\*
 & & $+24332v^7+20679v^6+11583v^5+3909v^4+702v^3+54v^2+v$ \\
 & 16 & $5v^{14}+70v^{13}+456v^{12}+1846v^{11}+5245v^{10}+11163v^9+18175v^8+22059v^7+18778v^6+10360v^5+3368v^4+559v^3+35v^2+v$ \\
 & 17 & $7v^{13}+91v^{12}+555v^{11}+2141v^{10}+5865v^9+11694v^8+16430v^7+15354v^6+8911v^5+2937v^4+471v^3+26v^2$ \\
 & 18 & $11v^{12}+144v^{11}+909v^{10}+3516v^9+8723v^8+13794v^7+13539v^6+7867v^5+2498v^4+374v^3+20v^2$ \\
 & 19 & $v^{13}+19v^{12}+172v^{11}+952v^{10}+3431v^9+8093v^8+12271v^7+11589v^6+6504v^5+2011v^4+296v^3+14v^2$ \\
 & 20 & $4v^{12}+78v^{11}+594v^{10}+2450v^9+6094v^8+9416v^7+8925v^6+4945v^5+1444v^4+182v^3+10v^2$ \\
 & 21 & $v^{14}+14v^{13}+91v^{12}+378v^{11}+1166v^{10}+2831v^9+5286v^8+7084v^7+6284v^6+3383v^5+995v^4+138v^3+6v^2$ \\
 & 22 & $v^{13}+13v^{12}+80v^{11}+325v^{10}+979v^9+2179v^8+3370v^7+3358v^6+1986v^5+634v^4+95v^3+4v^2$ \\
 & 23 & $v^{12}+12v^{11}+72v^{10}+292v^9+830v^8+1563v^7+1824v^6+1229v^5+431v^4+63v^3+2v^2$ \\
 & 24 & $2v^{11}+22v^{10}+118v^9+394v^8+834v^7+1061v^6+750v^5+259v^4+33v^3+v^2$ \\
 & 25 & $3v^{10}+28v^9+124v^8+315v^7+458v^6+356v^5+134v^4+18v^3$ \\
 & 26 & $3v^9+24v^8+84v^7+151v^6+138v^5+58v^4+9v^3$ \\
 & 27 & $3v^8+18v^7+43v^6+48v^5+24v^4+4v^3$ \\
 & 28 & $v^7+4v^6+6v^5+4v^4+v^3$ \\
\hline
$B_8$ & 0 & $v^8+8v^7+28v^6+56v^5+70v^4+56v^3+28v^2+8v+1$ \\
 & 1 & $7v^8+50v^7+153v^6+260v^5+265v^4+162v^3+55v^2+8v$ \\
 & 2 & $v^{10}+14v^9+92v^8+327v^7+682v^6+871v^5+684v^4+317v^3+77v^2+7v$ \\
 & 3 & $v^{11}+16v^{10}+112v^9+454v^8+1150v^7+1863v^6+1931v^5+1252v^4+479v^3+95v^2+7v$ \\
 & 4 & $v^{12}+15v^{11}+108v^{10}+486v^9+1471v^8+3015v^7+4128v^6+3690v^5+2070v^4+676v^3+110v^2+6v$ \\
 & 5 & $v^{13}+13v^{12}+92v^{11}+448v^{10}+1561v^9+3870v^8+6700v^7+7900v^6+6143v^5+3004v^4+853v^3+121v^2+6v$ \\
 & 6 & $2v^{13}+31v^{12}+229v^{11}+1065v^{10}+3401v^9+7627v^8+11942v^7+12788v^6+9067v^5+4047v^4+1042v^3+130v^2+5v$ \\
 & 7 & $v^{14}+18v^{13}+149v^{12}+771v^{11}+2814v^{10}+7535v^9+14748v^8+20617v^7+19928v^6+12771v^5+5109v^4+1158v^3+126v^2+5v$ \\
 & 8 & $v^{15}+16v^{14}+124v^{13}+621v^{12}+2266v^{11}+6414v^{10}+14299v^9+24504v^8+30923v^7+27414v^6+16226v^5+6012v^4+1263v^3+124v^2+3v$ \\
 & 9 & $3v^{14}+48v^{13}+378v^{12}+1922v^{11}+6888v^{10}+17710v^9+32338v^8+41105v^7+35391v^6+19841v^5+6811v^4+1293v^3+113v^2+3v$ \\
 & 10 & $v^{15}+20v^{14}+189v^{13}+1114v^{12}+4582v^{11}+13794v^{10}+30628v^9+49418v^8+56495v^7+44316v^6+22875v^5+7315v^4+1311v^3+106v^2+2v$ \\
 & 11 & $v^{17}+17v^{16}+137v^{15}+697v^{14}+2530v^{13}+7072v^{12}+16142v^{11}+31272v^{10}+51411v^9+68679v^8+69937v^7+50919v^6+24859v^5+7515v^4+1242v^3+90v^2+2v$ \\
 & 12 & $2v^{16}+33v^{15}+259v^{14}+1294v^{13}+4663v^{12}+13014v^{11}+29163v^{10}+52406v^9+73033v^8+75123v^7+54033v^6+25590v^5+7383v^4+1155v^3+80v^2+v$ \\
 & 13 & $6v^{15}+94v^{14}+700v^{13}+3304v^{12}+11101v^{11}+27994v^{10}+53590v^9+76467v^8+78518v^7+55502v^6+25562v^5+7094v^4+1035v^3+58v^2+v$ \\
 & 14 & $v^{16}+18v^{15}+161v^{14}+938v^{13}+3934v^{12}+12461v^{11}+30429v^{10}+57011v^9+79806v^8+80222v^7+55210v^6+24504v^5+6433v^4+865v^3+45v^2$ \\
 & 15 & $v^{17}+17v^{16}+138v^{15}+716v^{14}+2695v^{13}+7937v^{12}+19160v^{11}+38578v^{10}+63741v^9+82575v^8+79003v^7+52404v^6+22489v^5+5698v^4+731v^3+33v^2$ \\
 & 16 & $2v^{16}+32v^{15}+245v^{14}+1206v^{13}+4333v^{12}+12177v^{11}+27551v^{10}+49635v^9+68267v^8+67779v^7+45742v^6+19579v^5+4819v^4+576v^3+24v^2$ \\
 & 17 & $4v^{15}+60v^{14}+432v^{13}+2011v^{12}+6857v^{11}+18010v^{10}+36246v^9+53737v^8+55762v^7+38323v^6+16313v^5+3894v^4+439v^3+17v^2$ \\
 & 18 & $7v^{14}+101v^{13}+720v^{12}+3356v^{11}+11057v^{10}+25880v^9+42018v^8+45692v^7+31825v^6+13321v^5+3028v^4+316v^3+11v^2$ \\
 & 19 & $v^{16}+16v^{15}+123v^{14}+612v^{13}+2238v^{12}+6452v^{11}+15074v^{10}+28034v^9+39353v^8+38978v^7+25409v^6+10122v^5+2251v^4+239v^3+6v^2$ \\
 & 20 & $v^{15}+16v^{14}+125v^{13}+636v^{12}+2369v^{11}+6753v^{10}+14607v^9+22967v^8+24810v^7+17326v^6+7254v^5+1606v^4+145v^3+4v^2$ \\
 & 21 & $3v^{14}+44v^{13}+306v^{12}+1351v^{11}+4226v^{10}+9649v^9+15714v^8+17327v^7+12116v^6+4947v^5+1049v^4+98v^3+2v^2$ \\
 & 22 & $v^{14}+17v^{13}+132v^{12}+633v^{11}+2112v^{10}+5100v^9+8737v^8+10052v^7+7258v^6+3046v^5+665v^4+56v^3+v^2$ \\
 & 23 & $v^{13}+18v^{12}+142v^{11}+661v^{10}+2007v^9+4029v^8+5170v^7+4012v^6+1749v^5+377v^4+32v^3$ \\
 & 24 & $2v^{12}+29v^{11}+186v^{10}+696v^9+1629v^8+2354v^7+2001v^6+937v^5+214v^4+17v^3$ \\
 & 25 & $3v^{11}+36v^{10}+190v^9+556v^8+938v^7+882v^6+429v^5+95v^4+8v^3$ \\
 & 26 & $3v^{10}+31v^9+130v^8+276v^7+306v^6+169v^5+42v^4+3v^3$ \\
 & 27 & $3v^9+22v^8+62v^7+82v^6+50v^5+12v^4+v^3$ \\
 & 28 & $v^8+5v^7+9v^6+7v^5+2v^4$ \\
\hline
$C_8$ & 0 & $v^8+8v^7+28v^6+56v^5+70v^4+56v^3+28v^2+8v+1$ \\
 & 1 & $v^9+14v^8+71v^7+188v^6+295v^5+286v^4+169v^3+56v^2+8v$ \\
 & 2 & $v^{10}+17v^9+110v^8+371v^7+737v^6+906v^5+692v^4+315v^3+76v^2+7v$ \\
 & 3 & $v^{11}+17v^{10}+123v^9+501v^8+1254v^7+1996v^6+2034v^5+1301v^4+493v^3+97v^2+7v$ \\
 & 4 & $v^{12}+16v^{11}+120v^{10}+551v^9+1663v^8+3346v^7+4467v^6+3889v^5+2128v^4+680v^3+109v^2+6v$ \\
 & 5 & $v^{13}+15v^{12}+107v^{11}+500v^{10}+1687v^9+4116v^8+7057v^7+8244v^6+6350v^5+3080v^4+870v^3+123v^2+6v$ \\
 & 6 & $v^{13}+21v^{12}+188v^{11}+993v^{10}+3438v^9+8056v^8+12816v^7+13682v^6+9559v^5+4177v^4+1051v^3+129v^2+5v$ \\
 & 7 & $2v^{14}+30v^{13}+218v^{12}+1023v^{11}+3462v^{10}+8768v^9+16485v^8+22339v^7+21035v^6+13191v^5+5199v^4+1173v^3+128v^2+5v$ \\
 & 8 & $2v^{14}+36v^{13}+300v^{12}+1541v^{11}+5446v^{10}+13807v^9+25188v^8+32490v^7+28795v^6+16852v^5+6140v^4+1263v^3+121v^2+3v$ \\
 & 9 & $v^{16}+16v^{15}+121v^{14}+583v^{13}+2051v^{12}+5732v^{11}+13329v^{10}+25778v^9+39701v^8+45811v^7+37339v^6+20301v^5+6870v^4+1303v^3+115v^2+3v$ \\
 & 10 & $3v^{15}+48v^{14}+364v^{13}+1750v^{12}+6037v^{11}+15913v^{10}+32559v^9+50569v^8+57183v^7+44909v^6+23248v^5+7408v^4+1303v^3+102v^2+2v$ \\
 & 11 & $v^{16}+16v^{15}+127v^{14}+671v^{13}+2652v^{12}+8295v^{11}+20899v^{10}+41643v^9+62930v^8+68711v^7+51655v^6+25352v^5+7566v^4+1231v^3+91v^2+2v$ \\
 & 12 & $3v^{15}+52v^{14}+433v^{13}+2287v^{12}+8559v^{11}+23740v^{10}+48967v^9+73556v^8+77946v^7+56086v^6+26122v^5+7364v^4+1129v^3+75v^2+v$ \\
 & 13 & $v^{17}+17v^{16}+138v^{15}+715v^{14}+2683v^{13}+7882v^{12}+19096v^{11}+38978v^{10}+65736v^9+86995v^8+84977v^7+57777v^6+25751v^5+6961v^4+1000v^3+59v^2+v$ \\
 & 14 & $2v^{16}+33v^{15}+263v^{14}+1352v^{13}+5053v^{12}+14616v^{11}+33546v^{10}+60474v^9+82591v^8+81486v^7+55158v^6+24132v^5+6289v^4+856v^3+46v^2$ \\
 & 15 & $6v^{15}+94v^{14}+705v^{13}+3360v^{12}+11383v^{11}+28791v^{10}+54792v^9+76801v^8+76359v^7+51383v^6+22065v^5+5564v^4+717v^3+34v^2$ \\
 & 16 & $v^{17}+17v^{16}+136v^{15}+690v^{14}+2528v^{13}+7244v^{12}+17146v^{11}+34321v^{10}+56932v^9+73985v^8+70364v^7+45823v^6+19069v^5+4640v^4+569v^3+25v^2$ \\
 & 17 & $2v^{16}+32v^{15}+240v^{14}+1140v^{13}+3921v^{12}+10587v^{11}+23347v^{10}+41593v^9+56879v^8+55984v^7+37218v^6+15598v^5+3743v^4+434v^3+17v^2$ \\
 & 18 & $4v^{15}+60v^{14}+423v^{13}+1893v^{12}+6119v^{11}+15214v^{10}+29353v^9+42416v^8+43378v^7+29502v^6+12435v^5+2945v^4+330v^3+12v^2$ \\
 & 19 & $7v^{14}+99v^{13}+661v^{12}+2798v^{11}+8405v^{10}+18491v^9+29201v^8+31629v^7+22218v^6+9459v^5+2212v^4+234v^3+6v^2$ \\
 & 20 & $12v^{13}+164v^{12}+1041v^{11}+4053v^{10}+10556v^9+18595v^8+21609v^7+15855v^6+6884v^5+1583v^4+155v^3+4v^2$ \\
 & 21 & $v^{14}+14v^{13}+108v^{12}+584v^{11}+2270v^{10}+6219v^9+11616v^8+14180v^7+10743v^6+4706v^5+1064v^4+102v^3+2v^2$ \\
 & 22 & $2v^{13}+29v^{12}+208v^{11}+958v^{10}+2985v^9+6182v^8+8163v^7+6526v^6+2950v^5+677v^4+64v^3+v^2$ \\
 & 23 & $3v^{12}+44v^{11}+289v^{10}+1124v^9+2737v^8+4085v^7+3573v^6+1725v^5+411v^4+37v^3$ \\
 & 24 & $4v^{11}+56v^{10}+321v^9+1004v^8+1798v^7+1798v^6+953v^5+236v^4+20v^3$ \\
 & 25 & $5v^{10}+62v^9+291v^8+674v^7+803v^6+474v^5+124v^4+10v^3$ \\
 & 26 & $6v^9+59v^8+199v^7+299v^6+205v^5+58v^4+3v^3$ \\
 & 27 & $7v^8+44v^7+92v^6+77v^5+23v^4+v^3$ \\
 & 28 & $v^8+8v^7+21v^6+20v^5+5v^4$ \\
\hline
$D_8$ & 0 & $v^8+8v^7+28v^6+56v^5+70v^4+56v^3+28v^2+8v+1$ \\
 & 1 & $v^9+13v^8+64v^7+167v^6+260v^5+251v^4+148v^3+49v^2+7v$ \\
 & 2 & $v^{10}+15v^9+92v^8+304v^7+602v^6+746v^5+580v^4+272v^3+69v^2+7v$ \\
 & 3 & $v^{11}+14v^{10}+94v^9+377v^8+953v^7+1548v^6+1618v^5+1065v^4+416v^3+84v^2+6v$ \\
 & 4 & $2v^{11}+35v^{10}+246v^9+939v^8+2166v^7+3151v^6+2914v^5+1675v^4+562v^3+96v^2+6v$ \\
 & 5 & $v^{12}+19v^{11}+160v^{10}+778v^9+2358v^8+4593v^7+5807v^6+4728v^5+2403v^4+711v^3+105v^2+5v$ \\
 & 6 & $v^{14}+14v^{13}+93v^{12}+392v^{11}+1209v^{10}+2977v^9+5894v^8+8858v^7+9445v^6+6754v^5+3055v^4+801v^3+104v^2+5v$ \\
 & 7 & $3v^{13}+42v^{12}+280v^{11}+1202v^{10}+3677v^9+8080v^8+12412v^7+12906v^6+8770v^5+3704v^4+890v^3+101v^2+3v$ \\
 & 8 & $5v^{13}+75v^{12}+514v^{11}+2180v^{10}+6342v^9+12943v^8+18300v^7+17473v^6+10876v^5+4182v^4+906v^3+93v^2+3v$ \\
 & 9 & $v^{15}+15v^{14}+108v^{13}+510v^{12}+1802v^{11}+5044v^{10}+11247v^9+19327v^8+24342v^7+21367v^6+12437v^5+4513v^4+923v^3+86v^2+2v$ \\
 & 10 & $3v^{14}+45v^{13}+323v^{12}+1491v^{11}+4921v^{10}+11986v^9+21336v^8+26937v^7+23222v^6+13058v^5+4489v^4+848v^3+73v^2+2v$ \\
 & 11 & $11v^{13}+161v^{12}+1084v^{11}+4437v^{10}+12119v^9+22675v^8+28916v^7+24578v^6+13414v^5+4426v^4+795v^3+63v^2+v$ \\
 & 12 & $v^{16}+16v^{15}+120v^{14}+562v^{13}+1870v^{12}+4840v^{11}+10437v^{10}+19145v^9+28576v^8+32109v^7+25209v^6+12925v^5+3999v^4+659v^3+45v^2+v$ \\
 & 13 & $2v^{15}+30v^{14}+211v^{13}+938v^{12}+3026v^{11}+7682v^{10}+15712v^9+24986v^8+28984v^7+23016v^6+11759v^5+3570v^4+561v^3+32v^2$ \\
 & 14 & $4v^{14}+56v^{13}+369v^{12}+1556v^{11}+4798v^{10}+11272v^9+19640v^8+24046v^7+19573v^6+9983v^5+2933v^4+430v^3+24v^2$ \\
 & 15 & $8v^{13}+105v^{12}+665v^{11}+2718v^{10}+7716v^9+15082v^8+19625v^7+16335v^6+8264v^5+2344v^4+323v^3+16v^2$ \\
 & 16 & $v^{14}+17v^{13}+142v^{12}+750v^{11}+2752v^{10}+7221v^9+13315v^8+16567v^7+13236v^6+6406v^5+1730v^4+226v^3+11v^2$ \\
 & 17 & $2v^{13}+32v^{12}+246v^{11}+1170v^{10}+3694v^9+7747v^8+10529v^7+8957v^6+4548v^5+1276v^4+169v^3+6v^2$ \\
 & 18 & $v^{14}+14v^{13}+94v^{12}+412v^{11}+1333v^{10}+3302v^9+6063v^8+7701v^7+6289v^6+3065v^5+804v^4+95v^3+4v^2$ \\
 & 19 & $v^{13}+13v^{12}+85v^{11}+373v^{10}+1177v^9+2605v^8+3808v^7+3457v^6+1828v^5+508v^4+61v^3+2v^2$ \\
 & 20 & $2v^{12}+24v^{11}+139v^{10}+506v^9+1224v^8+1910v^7+1812v^6+976v^5+271v^4+32v^3+v^2$ \\
 & 21 & $3v^{11}+31v^{10}+153v^9+449v^8+807v^7+858v^6+509v^5+151v^4+16v^3$ \\
 & 22 & $3v^{10}+27v^9+109v^8+239v^7+292v^6+191v^5+61v^4+8v^3$ \\
 & 23 & $3v^9+21v^8+61v^7+90v^6+69v^5+25v^4+3v^3$ \\
 & 24 & $v^8+5v^7+10v^6+10v^5+5v^4+v^3$ \\
\hline
\end{longtable}
\end{center}

\bigskip

\noindent
\textbf{Acknowledgments:} We would like to
thank Eamonn A.\ O'Brien and William Unger for
verifying the values of the polynomials
from Table \ref{tab:coadjoint} for some small
good primes. Part of the research for this
paper was done during a stay of the second
author at the University of Birmingham
supported by a stipend of the Ruth and
Gert Massenberg Foundation, and also
during a visit of the first author to
Ruhr-University Bochum.  We thank the referees
for numerous comments, which have improved the exposition of this paper.

\end{document}